\documentclass[11pt]{amsart}

\usepackage{a4wide,enumerate,color,graphicx}
\usepackage{amsmath,bm}
\usepackage{scrextend} 
\allowdisplaybreaks
\usepackage[normalem]{ulem}
\usepackage[dvipsnames]{xcolor}
\usepackage{setspace}
\usepackage{amsbsy}
\usepackage{hyperref}
\usepackage[overload]{empheq}

\newcommand{\T}{\mathbb T}
\newcommand{\R}{{\mathbb R}} 
\newcommand{\del}{\partial}
\newcommand{\diver}{\operatorname{div}} 
\newcommand{\dx}{\textnormal{d}x}
\newcommand{\dt}{\textnormal{d}t}
\newcommand{\ds}{\textnormal{d}s}

\numberwithin{equation}{section}
\hypersetup{
    colorlinks=true,
    linkcolor=blue,
    filecolor=blue,      
    urlcolor=blue,
}

\theoremstyle{plain}
\newtheorem{theorem}{Theorem}
\newtheorem{proposition}[theorem]{Proposition}

\newtheorem{definition}[theorem]{Definition}

\newtheorem{remark}[theorem]{Remark}

\begin{document}

\title[The high--friction limit]{Alignment via friction for nonisothermal \\
multicomponent fluid systems}

\author{Stefanos Georgiadis}
\address[Stefanos Georgiadis]{
        King Abdullah University of Science and Technology, CEMSE Division, Thuwal, Saudi Arabia, 23955-6900. \\
         and  Institute for Analysis and Scientific Computing, Vienna University of Technology, Wiedner Hauptstra\ss e 8--10, 1040 Wien, Austria
        }
\email{stefanos.georgiadis@kaust.edu.sa}

\author{Athanasios E. Tzavaras}
\address[Athanasios E. Tzavaras]{King Abdullah University of Science and Technology, CEMSE Division, Thuwal, Saudi Arabia, 23955-6900.}
\email{athanasios.tzavaras@kaust.edu.sa}


\dedicatory{Dedicated to Shi Jin on the occasion of his 60th birthday with friendship and admiration}

\begin{abstract} 
The derivation of an approximate Class--I model for nonisothermal multicomponent systems of fluids, as the high-friction limit of a Class--II model is justified, by validating the Chapman--Enskog expansion performed from the Class--II model towards the Class--I model. The analysis proceeds by comparing two thermomechanical theories via relative entropy.
\end{abstract}

	
\maketitle
	


\section{Introduction}

Multicomponent systems of fluids, i.e. systems of fluids composed of several constituents, are common in nature and industry, with applications including gas separation, catalysis, sedimentation, dialysis, electrolysis, and ion transport \cite{WeKr00}. Due to the complexity of such systems, one distinguishes among different classes (or types) of models, depending on how much information is assumed on the modeling stage. More detailed models have the advantage that they describe the physical phenomena more accurately, but at the same time the number of unknowns makes it difficult to analyze the systems, implement numerical algorithms, and it is impossible to measure experimentally certain of the quantities involved. It is thus expedient to understand the methods of passage from more detailed models to less detailed ones, as the latter are simpler and easier to comprehend. Ample understanding exists in the literature concerning the mechanical aspects of modeling, considerations of consistency with thermodynamics, and modeling of dissipation mechanisms for multicomponent systems, see \cite{BoDr15,Gi99} for a detailed account and references therein. The focus here is on the effect of friction as a mechanism of passage to simplified models.

In the literature, Class--II models refers to models assuming detailed knowledge of the constituent velocities $(v_1,\dots,v_n)$ while Class--I models refers to those models using the barycentric velocity $v$ for the description of motion of the mixture. The reduction from Class--II to Class--I models proceeds via relaxation induced by friction from the constituent velocities to the common barycentric velocity. The mathematical theory of relaxation was initiated in the works \cite{CLL94,JX95} and in the context of problems with friction may lead to diffusion equations \cite{HL92, JPT98}.
In the context of reduction from Class--II to Class--I models, the relevant mechanism is one of alignment through friction to a common barycentric velocity and it is best captured through the Chapman--Enskog expansion  \cite{HuJuTz19,YYZ15}. In this work we describe this mechanism and provide a quantitative convergence result in a context of nonisothermal models.

To focus ideas, let $n$ be the number of components of the system. We consider the Class--II model consisting of $n$ mass balances, $n$ momentum balances and a single energy balance:
\begin{align}
\label{intro-mass2}
    \del_t\rho_i+\textnormal{div}(\rho_iv_i) &=0,
\\
\label{into-momentum2}
        \del_t (\rho_iv_i)+\textnormal{div}(\rho_iv_i\otimes v_i) &=\rho_ib_i-\nabla p_i-\frac{1}{\epsilon}\theta\sum_{j\not=i}b_{ij}\rho_i\rho_j(v_i-v_j),
\\
\label{intro-energy2}
   \del_t\Big(\rho e+\sum_{j=1}^n\frac{1}{2}\rho_jv_j^2\Big) & +\textnormal{div}\Big(\sum_{j=1}^n \Big(\rho_j e_j+p_j +\frac{1}{2}\rho_jv_j^2 \Big)v_j\Big) 
   \\
   &= \textnormal{div}\left(\kappa\nabla\theta\right) 
        +\sum_{j=1}^n\rho_jb_j\cdot v_j+\rho r,
        \nonumber
\end{align}
for $i\in\{1,\dots,n\}$. For simplicity we consider the problem on $\T^3\times[0,\infty)$, where $\T^3$ denotes the three dimensional torus, that is
with space--periodic boundary conditions. The same analysis can be performed in a bounded domain $\Omega$, with no--flux boundary conditions, i.e. 
\begin{equation}\label{nfbc}
 \rho_iv_i \cdot \nu = 0, \quad (\rho_iv_i \otimes v_i+p_i\mathbb{I}) \nu = 0, \quad \left(\sum_{j=1}^n \Big(\rho_j e_j+p_j +\frac{1}{2}\rho_jv_j^2 \Big)v_j-\kappa \nabla\theta\right) \cdot \nu = 0, 
 \end{equation}
 for all $i=1,\dots,n$, on the parabolic boundary $\del\Omega\times[0,\infty)$, where $\nu$ denotes the outward  normal to the boundary $\del\Omega$.

The variables of the model are the $n$ partial densities $\rho_i$, the $n$ partial velocities $v_i$ and the temperature $\theta$; the latter is assumed to be the same for all components. The remaining  quantities of the model are $b_i$ the external force exerted on the $i$--th component, $p_i$ the partial pressure and $\rho_ie_i$ the internal energy of the $i$--th component. 
Moreover, we define 
\[ \rho=\sum_{j=1}^n\rho_j \, , \quad v=\frac{1}{\rho}\sum_{j=1}^n\rho_jv_j \, , \quad p=\sum_{j=1}^np_j \, ,  \quad  \rho e=\sum_{j=1}^n\rho_je_j \, , \quad \rho b=\sum_{j=1}^n\rho_jb_j \]
to be respectively the total mass density $\rho$, the barycentric velocity $v$, the total pressure $p$, the (total) internal energy $\rho e$ and the total force $\rho b$. Finally, $\kappa$ is the thermal conductivity, $\rho r$ the radiative heat supply and $b_{ij}$ are positive and symmetric coefficients, modeling binary interactions between the components, with a strength that is measured by $\epsilon>0$.

The thermodynamics of the model is described by a set of constitutive relations, assuming that the Helmholtz free energy functions $\rho_i\psi_i$ are given, which are
\begin{align}
\label{intro-const1}
    \rho_i\psi_i &=\rho_i\psi_i(\rho_i,\theta),
\\
\label{intro-const2}
    \rho_i\psi_i &=\rho_i e_i-\rho_i\eta_i\theta,
\\
\label{intro-const3}
   \mu_i &=  (\rho_i\psi_i)_{\rho_i}
\\
\label{intro-const4}
   \rho_i\eta_i &= - (\rho_i\psi_i)_\theta
\end{align}
where $\rho_i\eta_i$ are the partial entropies, $\mu_i$ the chemical potentials and we denote partial derivatives by subscripts, e.g. $f_{\rho_i}$ stands for $\frac{\del f}{\del\rho_i}$. The Gibbs-Duhem relation, determining the pressure, is given by
\begin{equation} \label{intro-gd}
    \rho_i\psi_i+p_i=\rho_i\mu_i.
\end{equation}
Given equations \eqref{intro-mass2}--\eqref{intro-energy2} and the constitutive relations \eqref{intro-const1}--\eqref{intro-gd}, one can derive the balance of the total entropy $\rho\eta=\sum_{j=1}^n\rho_j\eta_j$, that reads: 
\begin{equation}\label{intro-entropy2}
\begin{split}
    \del_t(\rho\eta)+\textnormal{div}\left(\sum_{j=1}^n\rho_j\eta_jv_j\right) & = \textnormal{div}\left(\frac{1}{\theta}\kappa\nabla\theta\right)+\frac{1}{\theta^2}\kappa|\nabla\theta|^2 \\
    & \phantom{xxxx} + \frac{1}{2\epsilon}\sum_{i=1}^n\sum_{j=1}^nb_{ij}\rho_i\rho_j|v_i-v_j|^2+\frac{\rho r}{\theta}
\end{split}
\end{equation} and a derivation of \eqref{intro-entropy2} can be found in \cite[Section 5]{BoDr15} or \cite[Appendix C]{GeTz23}.

The above model was derived and studied in \cite{BoGrPa19, HuJuTz19} in the isothermal case and in \cite{BoDr15} in a more general setting including chemical reactions and viscosity, in the case $\epsilon=1$. It was further shown that, if we introduce the diffusional velocities $u_i:=v_i-v$, \eqref{intro-mass2}--\eqref{intro-energy2} can be approximated by a simplified model ignoring terms of order $|u_i|^2$, which contains only the barycentric velocity $v$ and the diffusional velocities $u_i$ (and not the partial velocities $v_i$). The approximate model, which is a Class--I model, reads: 
\begin{align}
\label{intro-mass1}
    \del_t\bar{\rho}_i +\textnormal{div}(\bar{\rho}_i\bar{v}) &=-\textnormal{div}(\bar{\rho}_i\bar{u}_i), \\
\label{intro-momentum1}
    \del_t(\bar{\rho}\bar{v})+\textnormal{div}(\bar{\rho} \bar{v}\otimes\bar{v}) &=\bar{\rho}\bar{b}-\nabla \bar{p}, \\
\label{intro-energy1}
\begin{split}
   \del_t \Big( \bar{\rho}\bar{e}+\frac{1}{2}\bar{\rho} \bar{v}^2\Big)+\textnormal{div}\Big((\bar{\rho} \bar{e}+\bar{p}+\frac{1}{2}\bar{\rho}\bar{v}^2)\bar{v}\Big) &= \textnormal{div}\Big(\bar{\kappa}\nabla\bar{\theta}-\sum_{j=1}^n(\bar{\rho}_j\bar{e}_j+\bar{p}_j)\bar{u}_j\Big) \\
    &\quad +\bar{\rho}\bar{r}+\bar{\rho}\bar{b}\cdot \bar{v}+\sum_{j=1}^n\bar{\rho}_j\bar{b}_j\cdot\bar{u}_j,
\end{split}
\end{align}
where $\bar{u}_i$ is determined by solving the Maxwell--Stefan system:
\begin{equation}
    \label{intro-ms}
    -\sum_{j\not=i}b_{ij}\bar{\theta}\bar{\rho}_i\bar{\rho}_j(\bar{u}_i-\bar{u}_j) = \epsilon\left(\frac{\bar{\rho}_i}{\bar{\rho}}(\bar{\rho}\bar{b}-\nabla\bar{p})-\bar{\rho}_i\bar{b}_i+\nabla \bar{p}_i\right),
\end{equation} subject to the constraint
\begin{equation}
    \label{intro-constr}
    \sum_{i=1}^n \bar{\rho}_i\bar{u}_i = 0,
\end{equation}
which ensures that the total mass density is conserved. The thermodynamics of the model is described by the same laws \eqref{intro-const1}--\eqref{intro-gd} and the entropy balance now takes the form
\begin{equation}\label{intro-entropy1}
\begin{split}
    \del_t(\bar{\rho}\bar{\eta})+\textnormal{div}(\bar{\rho}\bar{\eta}\bar{v}) & = \textnormal{div}\left(\frac{1}{\bar\theta}\bar{\kappa}\nabla\bar{\theta}-\sum_{j=1}^n\bar{\rho}_j\bar\eta_j\bar u_j\right)+\frac{1}{\bar{\theta}^2}\bar{\kappa}|\nabla\bar{\theta}|^2 \\
    & \phantom{xxxxxx} + \frac{1}{2\epsilon}\sum_{i=1}^n\sum_{j=1}^nb_{ij}\bar{\rho}_i\bar{\rho}_j|\bar{u}_i-\bar{u}_j|^2+\frac{\bar{\rho}\bar{r}}{\bar{\theta}}.
\end{split}
\end{equation}

The method used to derive the model \eqref{intro-mass1}--\eqref{intro-constr} in \cite{BoDr15} is of algebraic nature and tailored to the specific model.
 To provide a systematic method, the authors of \cite{HuJuTz19} view this problem as a relaxation process in the collisional time $\epsilon>0$: they rescale the last term in \eqref{into-momentum2} (which corresponds to friction) and investigate the limit $\epsilon\to0$ via a Chapman--Enskog expansion. As $\epsilon\to0$, the friction forces the partial velocities $v_i$ to align to a single velocity $v$,  the barycentric velocity describing the motion of the center of mass, and the  Class--I model 
 emerges as the $\mathcal{O}(\epsilon^2)$--approximation in the Chapman--Enskog expansion. The methodological approach of relative entropy was developed for hyperbolic conservation laws in \cite{Da79} and generalized to hyperbolic/parabolic systems in \cite{ChTz18}. It is used in \cite{HuJuTz19} in order to compare dissipative weak solutions to the Class--II model with strong solutions to the Class--I model thus validating the Chapman--Enskog expansion. 

Here, we employ a similar perspective in the context of non-isothermal models that include the balance of energy equation and entropy production inequality. As shown in \cite{GeTz23}, the Chapman--Enskog expansion applied to the Class--II model \eqref{intro-mass2}--\eqref{intro-energy2} produces at an $\mathcal{O}(\epsilon^2)$--approximation 
the Class--I system \eqref{intro-mass1}--\eqref{intro-constr}, with the diffusional velocities $u_i$ being of order $\mathcal{O}(\epsilon)$. 
We here develop a relative entropy identity for Class--II models and use it to justify the limiting process. This validates the high--friction limit in the weak--strong solution context, i.e. we compare a weak solution of the Class--II model with a strong solution of the Class--I model. It is shown that,  as $\epsilon\to0$, the weak solution converges to the strong one, in the relative entropy sense. Such a result assumes that a strong solution to the Class--I model truly exists, which has been established near equilibrium, \cite{Gi99}. More precisely, there exists a unique strong solution to Class--I models, which is local--in--time for general initial data and can be extended for all positive times for initial data close to an equilibrium state (see \cite[Chapter 9]{Gi99} for more details). 



\section{The relative entropy inequality for Class--II systems}

Let $\omega=(\rho_1,\dots,\rho_n,v_1,\dots,v_n,\theta)$ and 
$\bar{\omega}=(\bar{\rho}_1,\dots,\bar{\rho}_n,\bar{v}_1,\dots,\bar{v}_n,\bar{\theta})$ be two solutions of the Class--II system. Motivated by \cite{ChTz18} and \cite{GeTz23}, we define the relative entropy of $\omega$ and $\bar\omega$ as follows: 
\begin{equation}\label{relen}
    \mathcal{H}(\omega|\bar{\omega})(t)=\int_{\T^3}\left(\frac{1}{2}\sum_{i=1}^n\rho_i|v_i-\bar{v}_i|^2+\sum_{i=1}^n(\rho_i\psi_i)(\omega|\bar{\omega})+(\rho\eta-\bar{\rho}\bar{\eta})(\theta-\bar{\theta})\right)\dx
\end{equation}
 \[ (\rho_i\psi_i)(\omega|\bar{\omega})=\rho_i\psi_i-\bar{\rho}_i\bar{\psi}_i-(\bar{\rho}_i\bar{\psi}_i)_{\rho_i}(\rho_i-\bar{\rho}_i)-(\bar{\rho}_i\bar{\psi}_i)_{\theta}(\theta-\bar{\theta}). \] 
Throughout this paper, we use the convention $\bar f=f(\bar\omega)$ and therefore when we write $\bar\rho_i\bar\psi_i$, we mean $(\rho_i\psi_i)(\bar\omega)$, while $(\bar\rho_i\bar\psi_i)_{\rho_i}=\frac{\del(\rho_i\psi_i)}{\del\rho_i}\big|_{\omega=\bar\omega}$ and similarly $(\bar\rho_i\bar\psi_i)_\theta=\frac{\del(\rho_i\psi_i)}{\del\theta}\big|_{\omega=\bar\omega}$.

If $\rho_i\psi_i$ are $C^3$ on the set \[ U=\{0\leq\rho_i\leq M, \quad 0<\gamma\leq\rho\leq M, \quad 0<\gamma\leq\theta\leq M,\quad \textnormal{for some } \gamma,M>0\}, \] such that \begin{equation}\label{convexity}
    (\rho_i\psi_i)_{\rho_i\rho_i}>0 \quad \textnormal{and} \quad (\rho_i\psi_i)_{\theta\theta}<0,
\end{equation} then according to \cite{GeTz23}, there exists $C>0$ such that \begin{equation}\label{bound1}
    \sum_{i=1}^n(\rho_i\psi_i)(\omega|\bar\omega)+(\rho\eta-\bar{\rho}\bar{\eta})(\theta-\bar{\theta})\geq C\left(\sum_{i=1}^n|\rho_i-\bar{\rho}_i|^2+|\theta-\bar{\theta}|^2\right).
\end{equation} Therefore, the relative entropy defined in \eqref{relen} can serve as a measure of the distance between $\omega$ and $\bar{\omega}$. We note, that the conditions in \eqref{convexity} (known as Gibbs thermodynamic stability conditions) are natural in thermodynamics, as they follow from the basic assumptions that the temperature is a strictly positive quantity and that the energy is a convex function of the entropy, satisfied for example for the ideal gas (for more details see \cite[Appendix A]{GeTz23}).

\begin{remark} \label{rem}
The assumption $\rho_i\psi_i\in C^3(U)$ does not cover many free energy functions that are useful.
The interesting case of the ideal gas, 
$$
\psi_i=R_i\theta\log\rho_i-c_i\theta\log\theta \, ,  \quad \mbox{with $R_i,c_i > 0$ constants},
$$
is not $C^3$ unless both $\rho_i$ and $\theta$ are bounded away from zero, for instance  on the set 
\[ \tilde{U}=\{0<\gamma\leq\rho_i\leq M, \quad 0<\gamma\leq\rho\leq M, \quad 0<\gamma\leq\theta\leq M,\quad \textnormal{for some } \gamma,M>0\}. \] 
We refer to \cite[Chapter 5]{GeTz23} for a detailed explanation on which term is problematic.  
We note the popular model of the ideal gas violates the third law of thermodynamics \cite[Sec 1.10, Sec 3.4]{Callen85},
namely that the entropy vanishes when the absolute temperature goes to zero. 
\end{remark}


Let \[ H(\omega|\bar{\omega})=\sum_{i=1}^n\left[\frac{1}{2}\rho_i|v_i-\bar{v}_i|^2+(\rho_i\psi_i)(\omega|\bar{\omega})+(\rho_i\eta_i-\bar{\rho}_i\bar{\eta}_i)(\theta-\bar{\theta})\right] \] and \[ Q(\omega|\bar\omega)=\sum_{i=1}^n\left[\frac{1}{2}\rho_iv_i|v_i-\bar{v}_i|^2+(\rho_i\psi_i)(\omega|\bar{\omega})v_i+(\rho_i\eta_i-\bar{\rho}_i\bar{\eta}_i)(\theta-\bar{\theta})v_i+(p_i-\bar p_i)(v_i-\bar v_i)\right]. \] We want to obtain an identity of the form: 
\begin{equation}\label{identity}
    \del_tH(\omega|\bar\omega)+\diver Q(\omega|\bar\omega) = \textnormal{RHS}.
\end{equation} 

Carrying out the differentiations and using the equations \eqref{mas}, \eqref{mom'}, \eqref{ener'} and \eqref{entr'}, along with the thermodynamic relations \eqref{intro-const2}--\eqref{intro-gd}, we find that \[ \begin{split}
    & \textnormal{RHS} = -\frac{1}{2\epsilon}\sum_{i=1}^n\sum_{j=1}^n\bar{\theta}b_{ij}\rho_i\rho_j|(v_i-v_j)-(\bar{v}_i-\bar{v}_j)|^2 - \bar{\theta}\kappa|\nabla\log\theta-\nabla\log\bar{\theta}|^2 \\
    & - \sum_{i=1}^np_i(\omega|\bar\omega)\diver\bar{v}_i - \sum_{i=1}^n\rho_i[(v_i-\bar{v}_i)\cdot\nabla] \bar{v}_i\cdot(v_i-\bar{v}_i) + \sum_{i=1}^n\rho_i(b_i-\bar{b}_i)\cdot(v_i-\bar{v}_i) \\
    & - \sum_{i=1}^n(\rho_i\eta_i)(\omega|\bar\omega)(\del_t\bar{\theta}+\bar{v}_i\cdot\nabla \bar{\theta}) - \sum_{i=1}^n\rho_i(v_i-\bar{v}_i)\cdot\nabla\bar{\theta}(\eta_i-\bar{\eta}_i) +\left(\frac{\rho r}{\theta}-\frac{\bar\rho\bar r}{\bar\theta}\right)(\theta-\bar\theta) \\
    & + (\nabla\log\theta-\nabla\log\bar{\theta})\cdot\nabla\log\bar{\theta}(\theta\bar{\kappa}-\bar{\theta}\kappa) + \diver\left((\theta-\bar\theta)\left(\kappa\nabla\log\theta-\bar\kappa\nabla\log\bar\theta\right)\right) \\
    & - \frac{1}{\epsilon}\sum_{i=1}^n\sum_{j=1}^n\bar{\theta}b_{ij}\rho_i(v_i-\bar{v}_i)\cdot(\rho_j-\bar{\rho}_j)(\bar{v}_i-\bar{v}_j) + \frac{1}{\epsilon}\sum_{i=1}^n\sum_{j=1}^n(\theta-\bar{\theta})b_{ij}\rho_i\rho_j(v_i-\bar{v}_i)\cdot(\bar{v}_i-\bar{v}_j) \\
    & + \frac{1}{\epsilon}\sum_{i=1}^n\sum_{j=1}^n(\theta-\bar{\theta})b_{ij}(\rho_i-\bar{\rho}_i)\bar{\rho}_j(\bar{v}_i-\bar{v}_j)\cdot\bar{v}_i + \frac{1}{\epsilon}\sum_{i=1}^n\sum_{j=1}^n(\theta-\bar{\theta})b_{ij}\rho_i(\rho_j-\bar{\rho}_j)(\bar{v}_i-\bar{v}_j)\cdot\bar{v}_i,
\end{split} \] 
where the relative quantities are given by 
$$
\begin{aligned}
 p_i(\omega|\bar\omega) &=p_i-\bar{p}_i-(\bar{p}_i)_{\rho_i}(\rho_i-\bar{\rho}_i)-(\bar{p}_i)_{\theta}(\theta-\bar{\theta}) 
 \\
(\rho_i\eta_i)(\omega|\bar\omega) &=\rho_i\eta_i-\bar{\rho}_i\bar{\eta}_i-(\bar{\rho}_i\bar{\eta}_i)_{\rho_i}(\rho_i-\bar{\rho}_i)-(\bar{\rho}_i\bar{\eta}_i)_{\theta}(\theta-\bar{\theta}).
\end{aligned}
$$

The details of the computation proceed along the lines of \cite[Appendix D]{GeTz23} and are not presented here. While the computation
there concerns Class--I models, the formal computation can be adapted to the present case of Class--II modes in a straightforward way. A rigorous derivation of
\eqref{identity} between a weak and a strong solution is presented in section \ref{(3.3)}.



\section{Asymptotic Derivation of Class--I systems}

The goal is to show how Class--I models are derived as asymptotic high-friction limits in a context of nonisothermal models.
We proceed as follows: (i) First, we interpret a Class--I model as a Class--II system with error terms. (ii) Using the relative entropy formula we
compare  an exact solution to an approximate solution of a Class--II system. (iii) This needs to be done at some prescribed level of solutions; this is made precise in section \ref{(3.2)}, in which we give the definitions of weak and strong solutions. (iv) Finally,
the derivation of the convergence result is done in sections \ref{(3.3)} and \ref{(3.4)}.

\subsection{Reformulation of the Class--I model}

First, we embed a solution of a Class--I model into an approximate solution of a Class--II model.
The equations of the Class--II model contain the partial velocities $v_i$, while the equations of the Class--I model contain the barycentric velocity $v$ and the diffusional velocities $u_i$.

Let $(\bar \rho_1, ... , \bar \rho_n, \bar v, \bar \theta)$ be a solution of \eqref{intro-mass1}--\eqref{intro-energy1}. Then we set
$$
\bar v_i = \bar v + \bar u_i 
$$
and \eqref{intro-mass1}--\eqref{intro-energy1} and \eqref{intro-entropy1} read:

\begin{equation}\label{mas}
    \del_t\bar{\rho}_i+\diver(\bar{\rho}_i\bar{v}_i)=0
\end{equation}
\begin{equation}\label{mom}
    \del_t(\bar{\rho}\bar{v})+\diver(\bar{\rho}\bar{v}\otimes\bar{v})=\bar{\rho}\bar{b}-\nabla\bar{p}
\end{equation}
\begin{equation}\label{ener}
    \del_t\left(\bar{\rho}\bar{e}+\frac{1}{2}\bar{\rho} \bar{v}^2\right)+\diver\left(\sum_{j=1}^n(\bar{\rho}_j\bar{e}_j+\bar{p}_j)\bar{v}_j+\frac{1}{2}\bar{\rho} \bar{v}^2\bar{v}\right)=\diver(\bar{\kappa}\nabla\bar{\theta})+\sum_{j=1}^n\bar{\rho}_j\bar{b}_j\cdot \bar{v}_j+\bar{\rho}\bar{r}.
\end{equation}
\begin{equation}\label{entr'}
\begin{split}
    \del_t(\bar{\rho}\bar{\eta})+\textnormal{div}\left(\sum_{j=1}^n\bar{\rho}_j\bar{\eta}_j\bar{v}_j\right) & = \textnormal{div}\left(\frac{1}{\bar{\theta}}\bar{\kappa}\nabla\bar{\theta}\right)+\frac{1}{\bar{\theta}^2}\bar{\kappa}|\nabla\bar{\theta}|^2 \\
    & \phantom{xxxxxxx} + \frac{1}{2\epsilon}\sum_{i=1}^n\sum_{j=1}^nb_{ij}\bar{\rho}_i\bar{\rho}_j|\bar{v}_i-\bar{v}_j|^2+\frac{\bar{\rho}\bar{r}}{\bar{\theta}}.
\end{split}
\end{equation}

Next, we rewrite \eqref{mom} and \eqref{ener} in a form that resembles the equations of the Class--II model. 
We reformulate \eqref{mom} as:
\begin{equation}\label{mom'}
    \del_t (\bar{\rho}_i\bar{v}_i)+\textnormal{div}(\bar{\rho}_i\bar{v}_i\otimes \bar{v}_i)=\bar{\rho}_i\bar{b}_i-\nabla\bar p_i - \frac{\bar{\theta}}{\epsilon}\sum_{j\not=i}b_{ij}\bar{\rho}_i\bar{\rho}_j(\bar{v}_i-\bar{v}_j)+\bar{R}_i,   
\end{equation} where \[ \begin{split}
    \bar{R}_i & = \del_t(\bar{\rho}_i\bar{v})+\del_t(\bar{\rho}_i\bar{u}_i)+\diver(\bar{\rho}_i \bar{v}\otimes \bar{v})+\diver(\bar{\rho}_i \bar{v}\otimes \bar{u}_i)+\diver(\bar{\rho}_i \bar{u}_i\otimes \bar{v}) \\
    & \phantom{xx} +\diver(\bar{\rho}_i \bar{u}_i\otimes \bar{u}_i)-\bar{\rho}_i\bar{b}_i+\nabla\bar p_i+\frac{\bar{\theta}}{\epsilon}\sum_{j\ne i}b_{ij}\bar{\rho}_i\bar{\rho}_j(\bar{v}_i-\bar{v}_j).
\end{split} \] 
Using \eqref{intro-ms}, we obtain
\[ \begin{split}
    \bar{R}_i & = \del_t(\bar{\rho}_i\bar{v})+\del_t(\bar{\rho}_i\bar{u}_i)+\diver(\bar{\rho}_i \bar{v}\otimes \bar{v})+\diver(\bar{\rho}_i \bar{v}\otimes \bar{u}_i) \\
    & \phantom{xx} +\diver(\bar{\rho}_i \bar{u}_i\otimes \bar{v})
     +\diver(\bar{\rho}_i \bar{u}_i\otimes \bar{u}_i)-\frac{\bar{\rho}_i}{\bar{\rho}}(\bar{\rho}\bar{b}-\nabla\bar{p})
\end{split} \]
and by virtue of $ \del_t\bar{\rho}+\diver(\bar{\rho}\bar{v})=0$, we see that 
\[ \begin{split}
    \del_t(\bar{\rho}_i\bar{v})+\diver(\bar{\rho}_i \bar{v}\otimes \bar{v}) & = (\del_t\bar{\rho}_i+\diver(\bar{\rho}_i\bar{v}))\bar{v}+\bar{\rho}_i(\del_t\bar{v}+(\bar{v} \cdot \nabla) \bar{v}) \\
    & = -\diver(\bar{\rho}_i\bar{u}_i)\bar{v}+\frac{\bar{\rho}_i}{\bar{\rho}}(\del_t(\bar{\rho}\bar{v})+\diver(\bar{\rho} \bar{v}\otimes \bar{v})) \\
    & = -\diver(\bar{\rho}_i\bar{u}_i)\bar{v}+\frac{\bar{\rho}_i}{\bar{\rho}}(\bar{\rho} \bar{b}-\nabla \bar{p})
\end{split} \] and thus \begin{equation}\label{R_i}
    \begin{split}
    \bar{R}_i & = -\diver(\bar{\rho}_i\bar{u}_i)\bar v+\del_t(\bar{\rho}_i\bar{u}_i)+\diver(\bar{\rho}_i \bar{v}\otimes \bar{u}_i)+\diver(\bar{\rho}_i \bar{u}_i\otimes \bar{v})+\diver(\bar{\rho}_i \bar{u}_i\otimes \bar{u}_i).
\end{split}
\end{equation} 

Similarly, we reformulate \eqref{ener} as:
\begin{equation}\label{ener'}
\begin{split}
   \del_t\left(\bar{\rho} \bar{e}+\sum_{j=1}^n\frac{1}{2}\bar{\rho}_j\bar{v}_j^2\right) & +\textnormal{div}\left(\sum_{j=1}^n\left(\bar{\rho}_j \bar{e}_j+\bar{p}_j+\frac{1}{2}\bar{\rho}_j\bar{v}_j^2\right)\bar{v}_j\right) 
   \\
   & \phantom{xxxxxxxxxxxx} = \textnormal{div}\left(\bar{\kappa}\nabla\bar{\theta}\right)+\sum_{j=1}^n\bar{\rho}_j\bar{b}_j\cdot \bar{v}_j+\bar{\rho} \bar{r}+\bar{Q},
\end{split}
\end{equation} where 
\begin{equation}
\begin{split}\label{Q}
    \bar{Q} & = -\del_t\left(\frac{1}{2}\bar{\rho} \bar{v}^2\right)-\diver\left(\frac{1}{2}\bar{\rho} \bar{v}^2\bar{v}\right)+\del_t\left(\frac{1}{2}\sum_{j=1}^n\bar{\rho}_j\bar{v}_j^2\right)+\diver\left(\frac{1}{2}\sum_{j=1}^n\bar{\rho}_j\bar{v}_j^2\bar{v}_j\right) \\
    & = \del_t\left(\frac{1}{2}\sum_{j=1}^n\bar{\rho}_j\bar{u}_j^2\right)+\diver\left(\frac{1}{2}\sum_{j=1}^n\bar{\rho}_j\bar{u}_j^2\bar{u}_j\right)+\diver\left(\frac{3}{2}\sum_{j=1}^n\bar{\rho}_j\bar{u}_j^2\bar{v}\right)
\end{split}
\end{equation} because due to \eqref{intro-constr} \[ \frac{1}{2}\sum_{j=1}^n\bar{\rho}_j\bar{v}_j^2-\frac{1}{2}\bar{\rho}\bar{v}^2 = \frac{1}{2}\sum_{j=1}^n\bar{\rho}_j\bar{u}_j^2 \]  and \[ \frac{1}{2}\sum_{j=1}^n\bar{\rho}_j\bar{v}_j^2\bar{v}_j-\frac{1}{2}\bar{\rho}\bar{v}^2\bar{v} = \frac{1}{2}\sum_{j=1}^n\bar{\rho}_j\bar{u}_j^2\bar{u}_j+\frac{3}{2}\sum_{j=1}^n\bar{\rho}_j\bar{u}_j^2\bar{v}. \] 

The equations of the Class--I model 
are thus reformulated as equations of a Class--II model (namely equations \eqref{mas}, \eqref{mom'}, \eqref{ener'}), with the terms 
 $\bar{R}_i$ and $\bar Q$ given by \eqref{R_i} and \eqref{Q}, respectively. The latter are viewed as error terms. The Maxwell--Stefan system 
\begin{equation}
\begin{aligned}
    \label{ms'}
    -\sum_{j\not=i}b_{ij}\bar{\theta}\bar{\rho}_i\bar{\rho}_j(\bar{u}_i-\bar{u}_j)  &= \epsilon\left(\frac{\bar{\rho}_i}{\bar{\rho}}(\bar{\rho}\bar{b}-\nabla \bar{p})-\bar{\rho}_i\bar{b}_i+\nabla\bar{p}_i\right)
    \\
   \sum_{j = 1}^n \bar \rho_j \bar u_j &= 0
\end{aligned}
\end{equation} 
is uniquely solvable \cite{Gi91,HuJuTz19}, which implies $\bar u_i = \mathcal{O}(\epsilon)$ and thus for smooth solutions 
 $\bar{R}_i$ and $\bar Q$ are of order $\mathcal{O}(\epsilon)$ and $\mathcal{O}(\epsilon^2)$ respectively.

\subsection{Notions of solutions}\label{(3.2)}

In the following, we give the definitions of solutions that will be used. We use the notation $\omega = ((\rho_1,\dots,\rho_n,v_1,\dots,v_n,\theta)$.

\begin{definition}\label{defwksol}
    A function $(\rho_1,\dots,\rho_n,v_1,\dots,v_n,\theta)$ is called a weak solution of the Class--II model \eqref{intro-mass2}--\eqref{intro-energy2}, if for all $i\in\{1,\dots,n\}$: 
    \[ 0\leq\rho_i\in C^0([0,\infty);L^1(\T^3)), \quad \rho_iv_i\in C^0([0,\infty);L^1(\T^3;\R^3)), \] 
    \[ \rho_iv_i\otimes v_i\in L^1_{\textnormal{loc}}(\T^3\times[0,\infty);\R^3\times\R^3), ~~ p_i\in L^1_{\textnormal{loc}}(\T^3\times[0,\infty)), ~~ \rho_ib_i\in L^1_{\textnormal{loc}}(\T^3\times[0,\infty);\R^3), \]
    \[ 0<\theta\in C^0([0,\infty);L^1(\T^3)), \quad (\rho_ie_i+\frac{1}{2}\rho_iv_i^2)\in C^0([0,\infty);L^1(\T^3)), \] 
    \[ (\rho_ie_i+p_i+\frac{1}{2}\rho_iv_i^2)v_i\in L^1_{\textnormal{loc}}(\T^3\times[0,\infty);\R^3), \quad \kappa\nabla\theta\in L^1_{\textnormal{loc}}(\T^3\times[0,\infty);\R^3), \]
    \[ (\rho_ib_i\cdot v_i+\rho r)\in L^1_{\textnormal{loc}}(\T^3\times[0,\infty)), \quad \theta\sum_{j\not=i}b_{ij}\rho_i\rho_j(v_i-v_j) \in L^1_{\textnormal{loc}}(\T^3\times[0,\infty);\R^3) \]
    and $(\rho_1,\dots,\rho_n,v_1,\dots,v_n,\theta)$ solves for all test functions $\psi_i,\xi\in C_c^\infty([0,\infty);C^\infty(\T^3))$ and $\phi_i\in C_c^\infty([0,\infty);C^\infty(\T^3;\R^3))$:
    \begin{equation}\label{weak-mass}
        -\int_{\T^3}\rho_i(x,0)\psi_i(x,0)\dx-\int_0^\infty\int_{\T^3}\rho_i\del_t\psi_i\dx\dt-\int_0^\infty\int_{\T^3}\rho_iv_i\cdot\nabla\psi_i\dx\dt=0,
    \end{equation}
    \vspace{3mm}
    \begin{equation}\label{weak-momentum}
        \begin{split}
            & -\int_{\T^3}(\rho_iv_i)(x,0)\phi_i(x,0)\dx-\int_0^\infty\int_{\T^3}\rho_iv_i\cdot\del_t\phi_i\dx\dt \\
                        & \phantom{xx} -\int_0^\infty\int_{\T^3}(\rho_iv_i\otimes v_i+p_i\mathbb{I}):\nabla\phi_i\dx\dt \\
            & \phantom{xx} \quad = \int_0^\infty\int_{\T^3}\rho_ib_i\phi_i\dx\dt-\frac{1}{\epsilon}\int_0^\infty\int_{\T^3}\theta\sum_{j=1}^nb_{ij}\rho_i\rho_j(v_i-v_j)\phi_i\dx\dt
        \end{split}
    \end{equation}
    \vspace{3mm}
    and \begin{equation}\label{weak-energy}
        \begin{split}
            & -\int_0^\infty\int_{\T^3}(\rho e+\frac{1}{2}\sum_{j=1}^n\rho_jv_j^2)(x,0)\xi(x,0)\dx\dt-\int_0^\infty\int_{\T^3}(\rho e+\frac{1}{2}\sum_{j=1}^n\rho_jv_j^2)\del_t\xi\dx\dt \\
            & -\int_0^\infty\int_{\T^3}\sum_{j=1}^n(\rho_j e_j+p_j+\frac{1}{2}\rho_jv_j^2)v_j\cdot\nabla\xi\dx\dt = -\int_0^\infty\int_{\T^3}\kappa\nabla\theta\cdot\nabla\xi\dx\dt \\
            & +\int_0^\infty\int_{\T^3}(\sum_{j=1}^n\rho_jb_j\cdot v_j+\rho r)\xi\dx\dt.
        \end{split}
    \end{equation}
\end{definition}

\begin{definition}\label{defdisswksol}
    A function $(\rho_1,\dots,\rho_n,v_1,\dots,v_n,\theta)$ is called an entropy weak solution of the Class--II model \eqref{intro-mass2}--\eqref{intro-energy2}, if it is a weak solution 
    according to Definition \ref{defwksol} with the additional regularity
    \[ \rho\eta\in C^0([0,\infty);L^1(\T^3)), \quad  \rho_i\eta_iv_i\in L^1_{\textnormal{loc}}(\T^3\times[0,\infty);\R^3), \quad \mbox{$i\in\{1,\dots,n\}$} \]
    \[ \kappa\nabla\log\theta\in L^1_{\textnormal{loc}}(\T^3\times[0,\infty);\R^3), \quad \kappa|\nabla\log\theta|^2\in L^1_{\textnormal{loc}}(\T^3\times[0,\infty)), \]
    \[ \frac{\rho r}{\theta}\in L^1_{\textnormal{loc}}(\T^3\times[0,\infty)), \quad \sum_{i,j}b_{ij}\rho_i\rho_j|v_i-v_j|^2 \in L^1_{\textnormal{loc}}(\T^3\times[0,\infty)) \]
    that satisfies  the weak form of the integrated entropy inequality
     \begin{equation}\label{weak-entropy}
        \begin{split}
            & -\int_{\T^3}(\rho\eta)(x,0)\chi(x,0)\dx-\int_0^\infty\int_{\T^3}\rho\eta\del_t\chi\dx\dt-\int_0^\infty\int_{\T^3}\sum_{j=1}^n\rho_j\eta_jv_j\cdot\nabla\chi\dx\dt \\
            & \phantom{xx} \geq -\int_0^\infty\int_{\T^3}\frac{1}{\theta}\kappa\nabla\theta\cdot\nabla\chi\dx\dt+\int_0^\infty\int_{\T^3}\frac{1}{\theta^2}\kappa|\nabla\theta|^2\chi\dx\dt \\
            & \phantom{xxxx} + \frac{1}{2\epsilon}\sum_{i=1}^n\sum_{j=1}^n\int_0^\infty\int_{\T^3}b_{ij}\rho_i\rho_j|v_i-v_j|^2\chi\dx\dt+\int_0^\infty\int_{\T^3}\frac{\rho r}{\theta}\chi\dx\dt,
        \end{split}
    \end{equation} holds for all test functions $\chi\in C_c^\infty([0,\infty);C^\infty(\T^3))$, with $\chi\geq0$.
\end{definition}

\begin{definition}
    A function $(\bar{\rho}_1,\dots,\bar{\rho}_n,\bar{v}_1,\dots,\bar{v}_n,\bar{\theta})$ is called a strong solution of the Class--I model \eqref{intro-mass1}--\eqref{intro-constr}, if \eqref{intro-mass1}--\eqref{intro-entropy1} hold almost everywhere on $\T^3$ and for all $t>0$.
\end{definition}


\subsection{Derivation of the relative entropy inequality}\label{(3.3)}
Next, we derive the relative entropy inequality comparing a weak with a strong solution:

\begin{proposition}\label{prop}
    Let $\omega$ be an entropy weak solution of the Class--II model \eqref{intro-mass2}--\eqref{intro-energy2} and $\bar{\omega}$ a strong solution of the Class--I model \eqref{intro-mass1}--\eqref{intro-entropy1}. Then, the following relative entropy inequality
    \begin{equation}\label{relenin}
        \begin{split}
            & \mathcal{H}(\omega|\bar\omega)(t) + \frac{1}{2\epsilon}\sum_{i=1}^n\sum_{j=1}^n\int_0^t\int_{\T^3}\bar{\theta}b_{ij}\rho_i\rho_j|(v_i-v_j)-(\bar{v}_i-\bar{v}_j)|^2\dx\ds \\
            & + \int_0^t\int_{\T^3}\bar{\theta}\kappa|\nabla\log\theta-\nabla\log\bar{\theta}|^2\dx\ds \leq \mathcal{H}(\omega|\bar\omega)(0)-\sum_{i=1}^n\int_0^t\int_{\T^3}p_i(\omega|\bar\omega)\diver\bar{v}_i\dx\ds \\
            & -\sum_{i=1}^n\int_0^t\int_{\T^3}\rho_i[(v_i-\bar{v}_i)\cdot\nabla] \bar{v}_i\cdot(v_i-\bar{v}_i)\dx\ds+\sum_{i=1}^n\int_0^t\int_{\T^3}\rho_i(b_i-\bar{b}_i)\cdot(v_i-\bar{v}_i)\dx\ds \\
           & -\sum_{i=1}^n\int_0^t\int_{\T^3}(\rho_i\eta_i)(\omega|\bar\omega)(\del_s\bar{\theta}+\bar{v}_i\cdot\nabla \bar{\theta})\dx\ds +\sum_{i=1}^n\int_0^t\int_{\T^3}\bar{R}_i\cdot\bar{v}_i\dx\ds \\
            & -\sum_{i=1}^n\int_0^t\int_{\T^3}\rho_i(v_i-\bar{v}_i)\cdot\nabla\bar{\theta}(\eta_i-\bar{\eta}_i)\dx\ds -\sum_{i=1}^n\int_0^t\int_{\T^3}\frac{\rho_i}{\bar{\rho}_i}(v_i-\bar{v}_i)\cdot \bar{R}_i\dx\ds \\
            & +\int_0^t\int_{\T^3}(\nabla\log\theta-\nabla\log\bar{\theta})\cdot\nabla\log\bar{\theta}(\theta\bar{\kappa}-\bar{\theta}\kappa)\dx\ds-\int_0^t\int_{\T^3}\bar{Q}\dx\ds \\
            & -\frac{1}{\epsilon}\sum_{i=1}^n\sum_{j=1}^n\int_0^t\int_{\T^3}\bar{\theta}b_{ij}\rho_i(v_i-\bar{v}_i)\cdot(\rho_j-\bar{\rho}_j)(\bar{v}_i-\bar{v}_j)\dx\ds \\
            & +\frac{1}{\epsilon}\sum_{i=1}^n\sum_{j=1}^n\int_0^t\int_{\T^3}(\theta-\bar{\theta})b_{ij}\rho_i\rho_j(v_i-\bar{v}_i)\cdot(\bar{v}_i-\bar{v}_j)\dx\ds \\
            & +\frac{1}{\epsilon}\sum_{i=1}^n\sum_{j=1}^n\int_0^t\int_{\T^3}(\theta-\bar{\theta})b_{ij}(\rho_i-\bar{\rho}_i)\bar{\rho}_j(\bar{v}_i-\bar{v}_j)\cdot\bar{v}_i\dx\ds \\
            & +\frac{1}{\epsilon}\sum_{i=1}^n\sum_{j=1}^n\int_0^t\int_{\T^3}(\theta-\bar{\theta})b_{ij}\rho_i(\rho_j-\bar{\rho}_j)(\bar{v}_i-\bar{v}_j)\cdot\bar{v}_i\dx\ds \\
            & +\int_0^t\int_{\T^3}\left(\frac{\rho r}{\theta}-\frac{\bar\rho\bar r}{\bar\theta}\right)(\theta-\bar\theta)dxds.
        \end{split}
    \end{equation}
holds for every $t>0$,  where  $\bar{R}_i$ and $\bar Q$ are given by \eqref{R_i} and \eqref{Q}. 
\end{proposition}

\begin{remark}
The proof is done for periodic entropy weak solutions defined on $\T^3 \times (0,\infty)$. The same proof would carry over to solutions defined on a bounded domain 
$\Omega \times (0,\infty)$ that satisfy the no-flux boundary conditions \eqref{nfbc}. Concerning solutions of  Class-II models defined on the whole space
$\R^3 \times (0,\infty)$, the reader will note that the integrals in the relative entropy identity \eqref{relenin} are still well defined for classical solutions that approach the same
constant states $(\bar \rho, \bar v_i , \bar \theta)$ at infinity,  provided the functions decay sufficiently fast to the constant sate as $|x| \to \infty$. For such classical solutions one
can still derive the relative entropy inequality and it would be useful if the error terms $\bar R_i$ and $\bar Q$ are integrable.
\end{remark}

\begin{proof} Multiply \eqref{mas},\eqref{mom'},\eqref{ener'} and \eqref{entr'} by the test functions $\psi_i,\phi_i,\xi,\chi$ respectively, as in the weak formulation of the equations of the Class--II model, integrate them over $\T^3\times(0,\infty)$ and subtract them from \eqref{weak-mass}--\eqref{weak-entropy}, in order to obtain:

\[ -\int_{\T^3}(\rho_i-\bar{\rho}_i)(x,0)\psi_i(x,0)\dx-\int_0^\infty\int_{\T^3}(\rho_i-\bar{\rho}_i)\del_t\psi_i\dx\dt-\int_0^\infty\int_{\T^3}(\rho_iv_i-\bar{\rho}_i\bar{v}_i)\cdot\nabla\psi_i\dx\dt=0, \]
\vspace{3mm}
\begin{align*}
    & -\int_{\T^3}(\rho_iv_i-\bar{\rho}_i\bar{v}_i)(x,0)\phi_i(x,0)\dx-\int_0^\infty\int_{\T^3}(\rho_iv_i-\bar{\rho}_i\bar{v}_i)\del_t\phi_i\dx\dt \\
    & -\int_0^\infty\int_{\T^3}(\rho_iv_i\otimes v_i+p_i\mathbb{I}-\bar{\rho}_i\bar{v}_i\otimes\bar{v}_i-\bar{p}_i\mathbb{I}):\nabla\phi_i\dx\dt=\int_0^\infty\int_{\T^3}(\rho_ib_i-\bar{\rho}_i\bar{b}_i)\cdot\phi_i\dx\dt \\
    & -\frac{1}{\epsilon}\int_0^\infty\int_{\T^3}\left(\theta\sum_{j\ne i}b_{ij}\rho_i\rho_j(v_i-v_j)-\bar{\theta}\sum_{j\ne i}b_{ij}\bar{\rho}_i\bar{\rho}_j(\bar{v}_i-\bar{v}_j)\right)\cdot\phi_i\dx\dt-\int_0^\infty\int_{\T^3}\bar{R}_i\cdot\phi_i\dx\dt,
\end{align*}
\begin{align*}
    & -\int_{\T^3}\left(\rho e+\frac{1}{2}\sum_{j=1}^n\rho_jv_j^2-\bar{\rho}\bar{e}-\frac{1}{2}\sum_{j=1}^n\bar{\rho}_j\bar{v}_j^2\right)(x,0)\xi(x,0)\dx \\
    & -\int_0^\infty\int_{\T^3}\left(\rho e+\frac{1}{2}\sum_{j=1}^n\rho_jv_j^2-\bar{\rho}\bar{e}-\frac{1}{2}\sum_{j=1}^n\bar{\rho}_j\bar{v}_j^2\right)\del_t\xi\dx\dt+\int_0^\infty\int_{\T^3}(\kappa\nabla\theta-\bar{\kappa}\nabla\bar\theta)\cdot\nabla\xi\dx\dt \\
    & -\int_0^\infty\int_{\T^3}\sum_{j=1}^n\left((\rho_je_j+p_j+\frac{1}{2}\rho_jv_j^2)v_j-(\bar{\rho}_j\bar{e}_j+\bar{p}_j+\frac{1}{2}\bar{\rho}_j\bar{v}_j^2)\bar{v}_j\right)\cdot\nabla\xi\dx\dt \\
    & = \int_0^\infty\int_{\T^3}\left(\rho r+\sum_{j=1}^n\rho_jb_j\cdot v_j-\bar\rho\bar r-\sum_{j=1}^n\bar\rho_j\bar b_j\cdot\bar v_j\right)\xi\dx\dt-\int_0^\infty\int_{\T^3}\bar Q\xi\dx\dt
\end{align*} 
and 
\begin{align*}
    & -\int_{\T^3}(\rho \eta-\bar\rho\bar\eta)(x,0)\chi(x,0)\dx -\int_0^\infty\int_{\T^3}(\rho \eta-\bar\rho\bar\eta)\del_t\chi\dx\dt \\
    & -\int_0^\infty\int_{\T^3}\sum_{j=1}^n(\rho_j\eta_jv_j-\bar\rho_j\bar\eta_j\bar v_j)\cdot\nabla\chi\dx\dt \geq -\int_0^\infty\int_{\T^3}\left(\frac{1}{\theta}\kappa\nabla\theta-\frac{1}{\bar\theta}\bar{\kappa}\nabla\bar\theta\right)\cdot\nabla\chi\dx\dt \\
    & +\int_0^\infty\int_{\T^3}\left(\frac{1}{2\epsilon}\sum_{i=1}^n\sum_{j=1}^nb_{ij}\rho_i\rho_j|v_i-v_j|^2-\frac{1}{2\epsilon}\sum_{i=1}^n\sum_{j=1}^nb_{ij}\bar\rho_i\bar\rho_j|\bar v_i-\bar v_j|^2\right)\chi\dx\dt \\
    & + \int_0^\infty\int_{\T^3}\left(\frac{1}{\theta^2}\kappa|\nabla\theta|^2-\frac{1}{\bar\theta^2}\bar{\kappa}|\nabla\bar\theta|^2\right)\chi\dx\dt + \int_0^\infty\int_{\T^3}\left(\frac{\rho r}{\theta}-\frac{\bar\rho\bar r}{\bar\theta}\right)\chi\dx\dt.
\end{align*} 

We choose the test functions $\psi_i=\left(\bar\mu_i-\frac{1}{2}\bar v_i^2\right)\zeta, \quad \phi_i=\bar v_i\zeta, \quad \xi=-\zeta \quad \textnormal{and} \quad \chi=\bar\theta\zeta$
where \[ \zeta(s)=\begin{cases}
    1 & 0\leq s<t \\
    \frac{t-s}{\delta}+1 & t\leq s<t+\delta \\
    0 & s\geq t+\delta
\end{cases}, \] and let $\delta\to0$, to obtain: 

\begin{align*}
    & -\int_{\T^3}(\rho_i-\bar{\rho}_i)(x,0)\left(\bar\mu_i-\frac{1}{2}\bar v_i^2\right)(x,0)\dx-\int_0^t\int_{\T^3}(\rho_i-\bar{\rho}_i)\del_s\left(\bar\mu_i-\frac{1}{2}\bar v_i^2\right)\dx\ds \\
    & +\int_{\T^3}(\rho_i-\bar{\rho}_i)\left(\bar\mu_i-\frac{1}{2}\bar v_i^2\right)\dx\ds -\int_0^t\int_{\T^3}(\rho_iv_i-\bar{\rho}_i\bar{v}_i)\cdot\nabla\left(\bar\mu_i-\frac{1}{2}\bar v_i^2\right)\dx=0,
\end{align*} 
\begin{align*}
    & -\int_{\T^3}(\rho_iv_i-\bar{\rho}_i\bar{v}_i)(x,0)\bar v_i(x,0)\dx-\int_0^t\int_{\T^3}(\rho_iv_i-\bar{\rho}_i\bar{v}_i)\del_s\bar v_i\dx\ds+\int_{\T^3}(\rho_iv_i-\bar{\rho}_i\bar{v}_i)\bar v_i\dx \\
    & -\int_0^t\int_{\T^3}(\rho_iv_i\otimes v_i+p_i\mathbb{I} -\bar{\rho}_i\bar{v}_i\otimes\bar{v}_i-\bar{p}_i\mathbb{I}):\nabla\bar v_i\dx\ds=\int_0^t\int_{\T^3}(\rho_ib_i-\bar{\rho}_i\bar{b}_i)\cdot\bar v_i\dx\ds \\
    & -\frac{1}{\epsilon}\int_0^t\int_{\T^3}\left(\theta\sum_{j\ne i}b_{ij}\rho_i\rho_j(v_i-v_j)-\bar{\theta}\sum_{j\ne i}b_{ij}\bar{\rho}_i\bar{\rho}_j(\bar{v}_i-\bar{v}_j)\right)\cdot\bar v_i\dx\ds-\int_0^t\int_{\T^3}\bar{R}_i\cdot\bar v_i\dx\ds,
\end{align*} 
\begin{align*}
    & \int_{\T^3}\left(\rho e+\frac{1}{2}\sum_{j=1}^n\rho_jv_j^2-\bar{\rho}\bar{e}-\frac{1}{2}\sum_{j=1}^n\bar{\rho}_j\bar{v}_j^2\right)(x,0)\dx \\
    & -\int_{\T^3}\left(\rho e+\frac{1}{2}\sum_{j=1}^n\rho_jv_j^2-\bar{\rho}\bar{e}-\frac{1}{2}\sum_{j=1}^n\bar{\rho}_j\bar{v}_j^2\right)\dx \\
    & = -\int_0^t\int_{\T^3}\left(\rho r+\sum_{j=1}^n\rho_jb_j\cdot v_j-\bar\rho\bar r-\sum_{j=1}^n\bar\rho_j\bar b_j\cdot\bar v_j\right)\dx\ds+\int_0^t\int_{\T^3}\bar Q\dx\ds
\end{align*} 
\begin{align*}
    & -\int_{\T^3}(\rho \eta-\bar\rho\bar\eta)(x,0)\bar\theta(x,0)\dx -\int_0^t\int_{\T^3}(\rho \eta-\bar\rho\bar\eta)\del_s\bar\theta\dx\ds+\int_{\T^3}(\rho\eta-\bar\rho\bar\eta)\bar\theta\dx \\
    & -\int_0^t\int_{\T^3}\sum_{j=1}^n(\rho_j\eta_jv_j-\bar\rho_j\bar\eta_j\bar v_j)\cdot\nabla\bar\theta\dx\ds \geq -\int_0^t\int_{\T^3}\left(\frac{1}{\theta}\kappa\nabla\theta-\frac{1}{\bar\theta}\bar{\kappa}\nabla\bar\theta\right)\cdot\nabla\bar\theta\dx\ds \\
    & +\int_0^t\int_{\T^3}\left(\frac{1}{2\epsilon}\sum_{i=1}^n\sum_{j=1}^nb_{ij}\rho_i\rho_j|v_i-v_j|^2-\frac{1}{2\epsilon}\sum_{i=1}^n\sum_{j=1}^nb_{ij}\bar\rho_i\bar\rho_j|\bar v_i-\bar v_j|^2\right)\bar\theta\dx\ds \\
    & + \int_0^t\int_{\T^3}\left(\frac{1}{\theta^2}\kappa|\nabla\theta|^2-\frac{1}{\bar\theta^2}\bar{\kappa}|\nabla\bar\theta|^2\right)\bar\theta\dx\ds + \int_0^t\int_{\T^3}\left(\frac{\rho r}{\theta}-\frac{\bar\rho\bar r}{\bar\theta}\right)\bar\theta\dx\ds.
\end{align*}

Then, summing everything up and by virtue of the computation
\begin{align*}
    & -\sum_{i=1}^n(\rho_i-\bar\rho_i)\left(\bar\mu_i-\frac{1}{2}\bar v_i^2\right)-\sum_{i=1}^n(\rho_iv_i-\bar\rho_i\bar v_i)\bar v_i +\left(\rho e+\frac{1}{2}\sum_{i=1}^n\rho_iv_i^2-\bar\rho\bar e-\frac{1}{2}\sum_{i=1}^n\bar\rho_i\bar v_i^2\right) \\
    & -(\rho\eta-\bar\rho\bar\eta)\bar\theta = -\sum_{i=1}^n(\rho_i-\bar\rho_i)\bar\mu_i+\frac{1}{2}\sum_{i=1}^n(\rho_i\bar v_i^2-2\rho_iv_i\cdot\bar v_i+\rho_iv_i^2)+(\rho e-\bar\rho\bar e)-(\rho\eta-\bar\rho\bar\eta)\bar\theta \\
    & = \frac{1}{2}\sum_{i=1}^n\rho_i|v_i-\bar v_i|^2-\sum_{i=1}^n(\rho_i\psi_i)_{\rho_i}(\rho_i-\bar\rho_i)+\rho e-\bar\rho\bar e-\rho\eta\bar\theta+\bar\rho\bar\eta\bar\theta \\
    & = \frac{1}{2}\sum_{i=1}^n\rho_i|v_i-\bar v_i|^2+\sum_{i=1}^n(\rho_i\psi_i)(\omega|\bar\omega)+(\rho\eta-\bar\rho\bar\eta)(\theta-\bar\theta)=\mathcal{H}(\omega|\bar\omega)
\end{align*}
 one gets the inequality:
\[ \mathcal{H}(\omega|\bar\omega)(t)\leq \mathcal{H}(\omega|\bar\omega)(0)+I_1+I_2+I_3+I_4+I_5+I_6+I_7 \] where 
\begin{align*}
    I_1 & =-\sum_{i=1}^n\int_0^t\int_{\T^3}(\rho_i-\bar\rho_i)\cdot\del_s\left(\bar\mu_i-\frac{1}{2}\bar v_i^2\right)\dx\ds \\
    & \phantom{xx} -\sum_{i=1}^n\int_0^t\int_{\T^3}(\rho_iv_i-\bar\rho_i\bar v_i)\del_s\bar v_i\dx\ds-\int_0^t\int_{\T^3}(\rho\eta-\bar\rho\bar\eta)\del_s\bar\theta\dx\ds,
\end{align*}
\begin{align*}
    I_2 & = -\sum_{i=1}^n\int_0^t\int_{\T^3}(\rho_iv_i-\bar\rho_i\bar v_i)\cdot\nabla\left(\bar\mu_i-\frac{1}{2}\bar v_i^2\right)\dx\ds \\
    & \phantom{xx} -\sum_{i=1}^n\int_0^t\int_{\T^3}(\rho_iv_i\otimes v_i-\bar\rho_i\bar v_i\otimes\bar v_i):\nabla\bar v_i\dx\ds -\sum_{i=1}^n\int_0^t\int_{\T^3}(\rho_iv_i\eta_i-\bar\rho_i\bar v_i\bar\eta_i)\cdot\nabla\bar\theta\dx\ds,
\end{align*}
\begin{align*}
    I_3 & = \frac{1}{\epsilon}\sum_{i=1}^n\sum_{j=1}^n\int_0^t\int_{\T^3}\theta b_{ij}\rho_i\rho_j(v_i-v_j)\cdot\bar v_i\dx\ds-\frac{1}{\epsilon}\sum_{i=1}^n\sum_{j=1}^n\int_0^t\int_{\T^3}\bar\theta b_{ij}\bar \rho_i\bar \rho_j(\bar v_i-\bar v_j)\cdot\bar v_i\dx\ds \\
    & \phantom{xx}-\frac{1}{2\epsilon}\sum_{i=1}^n\sum_{j=1}^n\int_0^t\int_{\T^3}\bar\theta b_{ij}\rho_i\rho_j|v_i-v_j|^2\dx\ds+\frac{1}{2\epsilon}\sum_{i=1}^n\sum_{j=1}^n\int_0^t\int_{\T^3}\bar\theta b_{ij}\bar \rho_i\bar \rho_j|\bar v_i-\bar v_j|^2\dx\ds,
\end{align*}
$$
I_4=\int_0^t\int_{\T^3}\left(\frac{1}{\theta}\kappa\nabla\theta-\frac{1}{\bar\theta}\bar\kappa\nabla\bar\theta\right)\cdot\nabla\bar\theta\dx\ds-\int_0^t\int_{\T^3}\left(\frac{1}{\theta^2}\kappa|\nabla\theta|^2-\frac{1}{\bar\theta^2}\bar\kappa|\nabla\bar\theta|^2\right)\bar\theta\dx\ds,
$$
\begin{align*}
    I_5 & =-\sum_{i=1}^n\int_0^t\int_{\T^3}(\rho_ib_i-\bar\rho_i\bar b_i)\cdot\bar v_i\dx\ds \\
    & \phantom{xx} +\int_0^t\int_{\T^3}\left(\rho r+\sum_{i=1}^n\rho_ib_i\cdot v_i-\bar\rho\bar r-\sum_{i=1}^n\bar\rho_i\bar b_i\cdot\bar v_i\right)\dx\ds -\int_0^t\int_{\T^3}\left(\frac{\rho r}{\theta}-\frac{\bar\rho\bar r}{\bar\theta}\right)\bar\theta \dx\ds,
\end{align*}
$$
 I_6=- \sum_{i=1}^n\int_0^t\int_{\T^3}(p_i-\bar p_i)\diver\bar v_i\dx\ds 
$$
$$
I_7 = \sum_{i=1}^n\int_0^t\int_{\T^3}\bar R_i\cdot\bar v_i\dx\ds-\int_0^t\int_{\T^3}\bar Q\dx\ds
$$

and the plan is to rearrange the above terms in five steps.

\textit{Step 1: We rearrange the terms $I_1$, $I_2$ and $I_6$.} We start with $I_1$ and carry out the following calculation:
\[ \begin{split}
    I_1 & = -\int_0^t\int_{\T^3}\sum_{i=1}^n(\rho_i-\bar\rho_i)\left((\bar\mu_i)_{\rho_i}\del_s\bar\rho_i+(\bar\mu_i)_\theta\del_s\bar\theta\right)\dx\ds \\
    & \phantom{xx}-\int_0^t\int_{\T^3}\sum_{i=1}^n\rho_i(v_i-\bar v_i)\del_s\bar v_i\dx\ds-\int_0^t\int_{\T^3}\sum_{i=1}^n(\rho_i\eta_i-\bar\rho_i\bar\eta_i)\del_s\bar\theta\dx\ds \\
    & = -\int_0^t\int_{\T^3}\sum_{i=1}^n(\rho_i-\bar\rho_i)\left((\bar\mu_i)_{\rho_i}\del_s\bar\rho_i+(\bar\mu_i)_\theta\del_s\bar\theta\right)\dx\ds \\
    & \phantom{xx}-\int_0^t\int_{\T^3}\sum_{i=1}^n\rho_i(v_i-\bar v_i)\del_s\bar v_i\dx\ds-\int_0^t\int_{\T^3}\sum_{i=1}^n(\rho_i\eta_i)(\omega|\bar\omega)\del_s\bar\theta\dx\ds \\
    & \phantom{xx} -\int_0^t\int_{\T^3}\sum_{i=1}^n(\bar\rho_i\bar\eta_i)_{\rho_i}(\rho_i-\bar\rho_i)\del_s\bar\theta\dx\ds-\int_0^t\int_{\T^3}\sum_{i=1}^n(\bar\rho_i\bar\eta_i)_{\theta}(\theta-\bar\theta)\del_s\bar\theta\dx\ds
\end{split} \] and since \begin{equation}\label{aux3}
    (\rho_i\eta_i)_{\rho_i}=(-(\rho_i\psi_i)_\theta)_{\rho_i}=-((\rho_i\psi_i)_{\rho_i})_\theta=-(\mu_i)_\theta
\end{equation} we see that \[ \begin{split}
    I_1 & = -\int_0^t\int_{\T^3}\sum_{i=1}^n\rho_i(v_i-\bar v_i)\del_s\bar v_i\dx\ds -\int_0^t\int_{\T^3}\sum_{i=1}^n(\rho_i\eta_i)(\omega|\bar\omega)\del_s\bar\theta\dx\ds \\
    & \phantom{xx} -\int_0^t\int_{\T^3}\sum_{i=1}^n(\bar\rho_i\bar\eta_i)_{\theta}(\theta-\bar\theta)\del_s\bar\theta\dx\ds-\int_0^t\int_{\T^3}\sum_{i=1}^n(\rho_i-\bar\rho_i)(\bar\mu_i)_{\rho_i}\del_s\bar\rho_i\dx\ds \\
    & =: I_{11}+\cdots+I_{14},
\end{split} \] where \[ \begin{split}
    I_{13} & = -\int_0^t\int_{\T^3}\del_s(\bar\rho\bar\eta)(\theta-\bar\theta)\dx\ds+\int_0^t\int_{\T^3}\sum_{i=1}^n(\bar\rho_i\bar\eta_i)_{\rho_i}(\theta-\bar\theta)\del_s\bar\rho_i\dx\ds \\
    & = \int_0^t\int_{\T^3}\diver\left(\sum_{i=1}^n\bar\rho_i\bar\eta_i\bar v_i\right)(\theta-\bar\theta)+\int_0^t\int_{\T^3}\frac{1}{\bar{\theta}}\bar\kappa\nabla\bar\theta\cdot\nabla(\theta-\bar\theta)\dx\ds \\
    & \phantom{xx} -\int_0^t\int_{\T^3}\frac{1}{\bar\theta^2}\bar\kappa|\nabla\bar\theta|^2(\theta-\bar\theta)\dx\ds-\frac{1}{2\epsilon}\int_0^t\int_{\T^3}\sum_{i=1}^n\sum_{j=1}^nb_{ij}\bar\rho_i\bar\rho_j|\bar v_i-\bar v_j|^2(\theta-\bar\theta)\dx\ds \\
    & \phantom{xx} -\int_0^t\int_{\T^3}\frac{\bar\rho\bar r}{\bar\theta}(\theta-\bar\theta)\dx\ds- \int_0^t\int_{\T^3}\sum_{i=1}^n(\bar\rho_i\bar\eta_i)_{\rho_i}(\theta-\bar\theta)\nabla\bar\rho_i\cdot\bar v_i\dx\ds \\
    & \phantom{xx} - \int_0^t\int_{\T^3}\sum_{i=1}^n(\bar\rho_i\bar\eta_i)_{\rho_i}(\theta-\bar\theta)\bar\rho_i\diver\bar v_i\dx\ds=:I_{131}+\cdots+I_{137}
\end{split} \] and we have used \eqref{mas} and \eqref{entr'} and an integration by parts in the term $I_{132}$. 

Moreover, \[ I_{131}=\int_0^t\int_{\T^3}\sum_{i=1}^n\nabla(\bar\rho_i\bar\eta_i)\cdot\bar v_i(\theta-\bar\theta)\dx\ds+\int_0^t\int_{\T^3}\sum_{i=1}^n\bar\rho_i\bar\eta_i\diver\bar v_i(\theta-\bar\theta)\dx\ds=:I_{1311}+I_{1312}. \]

Again using \eqref{mas}, \[ \begin{split}
    I_{14} & = \int_0^t\int_{\T^3}\sum_{i=1}^n(\rho_i-\bar\rho_i)(\bar\mu_i)_{\rho_i}\nabla\bar\rho_i\cdot\bar v_i\dx\ds+\int_0^t\int_{\T^3}\sum_{i=1}^n(\rho_i-\bar\rho_i)(\bar\mu_i)_{\rho_i}\bar\rho_i\diver\bar v_i\dx\ds \\
    & = \int_0^t\int_{\T^3}\sum_{i=1}^n(\rho_i-\bar\rho_i)\nabla\bar\mu_i\cdot\bar v_i\dx\ds-\int_0^t\int_{\T^3}\sum_{i=1}^n(\rho_i-\bar\rho_i)(\bar\mu_i)_\theta\nabla\bar\theta\cdot\bar v_i\dx\ds \\
    & \phantom{xx} +\int_0^t\int_{\T^3}\sum_{i=1}^n(\rho_i-\bar\rho_i)(\bar\mu_i)_{\rho_i}\bar\rho_i\diver\bar v_i\dx\ds=:I_{141}+\cdots+I_{143}.
\end{split} \]

We now write $I_2$ as
\[ \begin{split}
    I_2  
    & = -\int_0^t\int_{\T^3}\sum_{i=1}^n(\rho_iv_i-\bar\rho_i\bar v_i)\cdot\nabla\bar\mu_i\dx\ds -\int_0^t\int_{\T^3}\sum_{i=1}^n\rho_i[(v_i-\bar{v}_i)\cdot\nabla]\bar v_i\cdot(v_i-\bar v_i) \\
    & \phantom{xx} -\int_0^t\int_{\T^3}\sum_{i=1}^n\rho_i[(v_i - \bar v_i)\cdot\nabla]\bar v_i\cdot \bar v_i\dx\ds -\int_0^t\int_{\T^3}\sum_{i=1}^n(\rho_i\eta_iv_i-\bar\rho_i\bar\eta_i\bar v_i)\cdot\nabla\bar\theta\dx\ds \end{split} \] 
and if we add and subtract the term with the relative pressure:
    \[ \begin{split}
    I_2 & = -\int_0^t\int_{\T^3}\sum_{i=1}^n(\rho_iv_i-\bar\rho_i\bar v_i)\cdot\nabla\bar\mu_i\dx\ds -\int_0^t\int_{\T^3}\sum_{i=1}^n\rho_i[(v_i-\bar{v}_i)\cdot\nabla]\bar v_i\cdot(v_i-\bar v_i) \\
    & \phantom{xx} -\int_0^t\int_{\T^3}\sum_{i=1}^n\rho_i[(v_i-\bar v_i)\cdot\nabla]\bar v_i\cdot\bar v_i\dx\ds -\int_0^t\int_{\T^3}\sum_{i=1}^np_i(\omega|\bar\omega)\diver\bar v_i\dx\ds \\
    & \phantom{xx} -\int_0^t\int_{\T^3}\sum_{i=1}^n(\bar p_i)_{\rho_i}(\rho_i-\bar\rho_i)\diver\bar v_i\dx\ds-\int_0^t\int_{\T^3}\sum_{i=1}^n(\bar p_i)_\theta(\theta-\bar\theta)\diver\bar v_i\dx\ds \\
    & \phantom{xx}-\int_0^t\int_{\T^3}\sum_{i=1}^n(\rho_i\eta_iv_i-\bar\rho_i\bar\eta_i\bar v_i)\cdot\nabla\bar\theta\dx\ds+\int_0^t\int_{\T^3}\sum_{i=1}^n(p_i-\bar p_i)\diver\bar v_i\dx\ds \\
    & =: I_{21}+\cdots+I_{28},
\end{split} \] 

where $I_{28}$ cancels out with $I_6$, while \[ \begin{split}
    I_{27} & = -\int_0^t\int_{\T^3}\sum_{i=1}^n(\rho_i\eta_i-\bar\rho_i\bar\eta_i)\bar v_i\cdot\nabla\bar\theta\dx\ds-\int_0^t\int_{\T^3}\sum_{i=1}^n\rho_i\eta_i(v_i-\bar v_i)\cdot\nabla\bar\theta\dx\ds \\
    & = -\int_0^t\int_{\T^3}\sum_{i=1}^n(\rho_i\eta_i)(\omega|\bar\omega)\bar v_i\cdot\nabla\bar\theta\dx\ds-\int_0^t\int_{\T^3}\sum_{i=1}^n(\bar\rho_i\bar\eta_i)_{\rho_i}(\rho_i-\bar\rho_i)\bar v_i\cdot\nabla\bar\theta\dx\ds \\
    & \phantom{xx} -\int_0^t\int_{\T^3}\sum_{i=1}^n(\bar\rho_i\bar\eta_i)_{\theta}(\theta-\bar\theta)\bar v_i\cdot\nabla\bar\theta\dx\ds-\int_0^t\int_{\T^3}\sum_{i=1}^n\rho_i\eta_i(v_i-\bar v_i)\cdot\nabla\bar\theta\dx\ds \\
    & =: I_{271}+\cdots+I_{274}
\end{split} \] and thus $I_{272}$ cancels out with $I_{142}$ and $I_{1311}$ cancels out with $I_{136}$ and $I_{273}$. Furthermore, due to \begin{equation}\label{gd-grad}
    \nabla p_i=\rho_i\nabla\mu_i+\rho_i\eta_i\nabla\theta
\end{equation} which can be obtained by applying the gradient operator to \eqref{intro-gd} and using \eqref{intro-const3} and \eqref{intro-const4}, we have

\[ \begin{split}
    I_{21} & = -\int_0^t\int_{\T^3}\sum_{i=1}^n\nabla\bar\mu_i\cdot(\rho_i-\bar\rho_i)v_i\dx\ds-\int_0^t\int_{\T^3}\sum_{i=1}^n\nabla\bar\mu_i\cdot\bar\rho_i(v_i-\bar v_i)\dx\ds \\
    & = -\int_0^t\int_{\T^3}\sum_{i=1}^n\nabla\bar\mu_i\cdot(\rho_i-\bar\rho_i)(v_i-\bar v_i)\dx\ds-\int_0^t\int_{\T^3}\sum_{i=1}^n\nabla\bar\mu_i\cdot(\rho_i-\bar\rho_i)\bar v_i\dx\ds \\
    & \phantom{xx} + \int_0^t\int_{\T^3}\sum_{i=1}^n\bar\rho_i\bar\eta_i\nabla\bar\theta\cdot(v_i-\bar v_i)\dx\ds-\int_0^t\int_{\T^3}\sum_{i=1}^n\nabla\bar p_i\cdot(v_i-\bar v_i)\dx\ds \\
    & =: I_{211}+\cdots+I_{214},
\end{split} \] where $I_{212}$ cancels out with $I_{141}$ and \[ I_{213}+I_{274}=-\int_0^t\int_{\T^3}\sum_{i=1}^n(\rho_i\eta_i-\bar\rho\bar\eta_i)(v_i-\bar v_i)\cdot\nabla\bar\theta\dx\ds. \]

Regarding $I_{11}$, using \eqref{mom'} we get \[ \begin{split}
    I_{11} & = \int_0^t\int_{\T^3}\sum_{i=1}^n\rho_i[(v_i-\bar v_i)\cdot \nabla]\bar v_i \cdot \bar v_i \dx\ds -\int_0^t\int_{\T^3}\sum_{i=1}^n\rho_i(v_i-\bar v_i)\cdot\bar b_i\dx\ds \\
    & \phantom{xx} + \int_0^t\int_{\T^3}\sum_{i=1}^n\frac{\rho_i}{\bar\rho_i}(v_i-\bar v_i)\cdot\nabla\bar p_i\dx\ds-\int_0^t\int_{\T^3}\sum_{i=1}^n\frac{\rho_i}{\bar\rho_i}(v_i-\bar v_i)\cdot\bar R_i\dx\ds \\
    & \phantom{xx} +\frac{1}{\epsilon}\int_0^t\int_{\T^3}\sum_{i=1}^n\sum_{j=1}^n\rho_i(v_i-\bar v_i)\bar\theta b_{ij}\bar\rho_j(\bar v_i-\bar v_j)\dx\ds =:I_{111}+\cdots+I_{115}.
\end{split} \] Notice that $I_{111}$ cancels out with $I_{23}$ and \[ I_{214}+I_{113}=\int_0^t\int_{\T^3}\sum_{i=1}^n\frac{1}{\bar\rho_i}\nabla\bar p_i\cdot(v_i-\bar v_i)(\rho_i-\bar\rho_i), \] which combined with $I_{211}$ gives, due to \eqref{gd-grad}, \[ I_{214}+I_{113}+I_{211}=\int_0^t\int_{\T^3}\sum_{i=1}^n\bar\eta_i(\rho_i-\bar\rho_i)(v_i-\bar v_i)\cdot\nabla\bar\theta\dx\ds \] and hence \[ I_{214}+I_{113}+I_{211}+I_{213}+I_{274}=-\int_0^t\int_{\T^3}\sum_{i=1}^n\rho_i(v_i-\bar v_i)\cdot\nabla\bar\theta(\eta_i-\bar\eta_i)\dx\ds. \]

Finally, \eqref{intro-const1}--\eqref{intro-const4} imply 
\begin{align}\label{aux1}
    (p_i)_{\rho_i} &=\rho_i(\mu_i)_{\rho_i}
\\
\label{aux2}
    (p_i)_\theta &=\rho_i\eta_i+\rho_i(\mu_i)_\theta
\end{align}
and due to \eqref{aux2}, \[ I_{1312}+I_{26}=-\int_0^t\int_{\T^3}\sum_{i=1}^n\bar\rho_i(\bar\mu_i)_\theta(\theta-\bar\theta)\diver\bar v_i\dx\ds \] which cancels out with $I_{137}$, because of \eqref{aux3}, while due to \eqref{aux1}, $I_{25}$ cancels out with $I_{143}$.

Putting together $I_1,I_2$ and $I_6$, we get 
 \begin{equation}\label{1+2+6}
     \begin{split}
         & I_1+I_2+I_6 
         = -\sum_{i=1}^n\int_0^t\int_{\T^3}p_i(\omega|\bar\omega)\diver\bar{v}_i\dx\ds -\sum_{i=1}^n\int_0^t\int_{\T^3}\frac{\rho_i}{\bar{\rho}_i}(v_i-\bar{v}_i)\cdot \bar{R}_i\dx\ds \\
         & \phantom{xx} - \sum_{i=1}^n\int_0^t\int_{\T^3}\rho_i[(v_i-\bar{v}_i)\cdot\nabla] \bar{v}_i\cdot(v_i-\bar{v}_i)\dx\ds - \int_0^t\int_{\T^3}\sum_{i=1}^n\rho_i(v_i-\bar v_i)\cdot\bar b_i\dx\ds \\
         & \phantom{xx} -\sum_{i=1}^n\int_0^t\int_{\T^3}(\rho_i\eta_i)(\omega|\bar\omega)(\del_s\bar{\theta}+\bar{v}_i\cdot\nabla \bar{\theta})\dx\ds +\int_0^t\int_{\T^3}\frac{1}{\bar\theta}\bar\kappa\nabla\bar\theta\cdot\nabla(\theta-\bar\theta)\dx\ds \\
         & \phantom{xx} -\sum_{i=1}^n\int_0^t\int_{\T^3}\rho_i(v_i-\bar{v}_i)\cdot\nabla\bar{\theta}(\eta_i-\bar{\eta}_i)\dx\ds - \int_0^t\int_{\T^3}\frac{1}{\bar\theta^2}\bar\kappa|\nabla\bar\theta|^2(\theta-\bar\theta)\dx\ds \\
         & \phantom{xx} +\frac{1}{\epsilon}\int_0^t\int_{\T^3}\sum_{i=1}^n\sum_{j=1}^n\rho_i(v_i-\bar v_i)\bar\theta b_{ij}\bar\rho_j(\bar v_i-\bar v_j)\dx\ds -\int_0^t\int_{\T^3}\frac{\bar\rho\bar r}{\bar\theta}(\theta-\bar\theta)\dx\ds \\
         & \phantom{xx} -\frac{1}{2\epsilon}\int_0^t\int_{\T^3}\sum_{i=1}^n\sum_{j=1}^nb_{ij}\bar\rho_i\bar\rho_j|\bar v_i-\bar v_j|^2(\theta-\bar\theta)\dx\ds.
     \end{split}
 \end{equation}


The rest of the steps consist in combining the terms on the right--hand--side of \eqref{1+2+6} with $I_3, I_4, I_5$ and $I_7$:

\textit{Step 2: Combine $I_3$ with the last and third--to--last term on the right--hand--side of \eqref{1+2+6}.}

We start by noticing that
\[ I_3 = \frac{1}{\epsilon}\sum_{i=1}^n\sum_{j=1}^n\int_0^t\int_{\T^3}\theta b_{ij}\rho_i\rho_j(v_i-v_j)\cdot\bar v_i\dx\ds - \frac{1}{2\epsilon}\sum_{i=1}^n\sum_{j=1}^n\int_0^t\int_{\T^3}\bar\theta b_{ij}\rho_i\rho_j|v_i-v_j|^2\dx\ds. \]

\noindent
The reason is that due to the symmetry of $b_{ij}$

\[
-\frac{1}{\epsilon}\sum_{i=1}^n\sum_{j=1}^n\bar\theta b_{ij}\bar \rho_i\bar \rho_j(\bar v_i-\bar v_j)\cdot\bar v_i\dx\ds = -\frac{1}{2\epsilon}\sum_{i=1}^n\sum_{j=1}^n\bar\theta b_{ij}\bar \rho_i\bar \rho_j|\bar v_i-\bar v_j|^2\dx\ds \]

\noindent
and thus the second and fourth terms of $I_3$ cancel out with each other. Therefore, we have:

\[ \begin{split}
     F & := I_3 +\frac{1}{\epsilon}\int_0^t\int_{\T^3}\sum_{i=1}^n\sum_{j=1}^n\rho_i(v_i-\bar v_i)\bar\theta b_{ij}\bar\rho_j(\bar v_i-\bar v_j)\dx\ds \\
    & \phantom{xxxxxxxxxxx} -\frac{1}{2\epsilon}\int_0^t\int_{\T^3}\sum_{i=1}^n\sum_{j=1}^nb_{ij}\bar\rho_i\bar\rho_j|\bar v_i-\bar v_j|^2(\theta-\bar\theta)\dx\ds \\
    & = \frac{1}{\epsilon}\int_0^t\int_{\T^3}\sum_{i=1}^n\sum_{j=1}^n\theta b_{ij}\rho_i\rho_j(v_i-v_j)\cdot\bar v_i\dx\ds -\frac{1}{2\epsilon}\int_0^t\int_{\T^3}\sum_{i=1}^n\sum_{j=1}^n\bar\theta b_{ij}\rho_i\rho_j|v_i-v_j|^2\dx\ds \\
    & +\frac{1}{\epsilon}\int_0^t\int_{\T^3}\sum_{i=1}^n\sum_{j=1}^n\rho_i(v_i-\bar v_i)\bar\theta b_{ij}\bar\rho_j(\bar v_i-\bar v_j)\dx\ds-\frac{1}{2\epsilon}\int_0^t\int_{\T^3}\sum_{i=1}^n\sum_{j=1}^nb_{ij}\theta\bar\rho_i\bar\rho_j|\bar v_i-\bar v_j|^2\dx\ds \\
    & +\frac{1}{2\epsilon}\int_0^t\int_{\T^3}\sum_{i=1}^n\sum_{j=1}^nb_{ij}\bar\theta\bar\rho_i\bar\rho_j|\bar v_i-\bar v_j|^2\dx\ds =: F_1+\cdots+F_5
\end{split} \]
and we start by collecting only the terms that are multiplied by $\bar\theta$:
\[ \begin{split}
    & F_2 + F_5 + F_3 = -\frac{1}{2\epsilon}\int_0^t\int_{\T^3}\sum_{i=1}^n\sum_{j=1}^n\bar\theta b_{ij}\rho_i\rho_j|v_i-v_j|^2\dx\ds \\
    & +\frac{1}{2\epsilon}\int_0^t\int_{\T^3}\sum_{i=1}^n\sum_{j=1}^n\bar\theta b_{ij}\bar\rho_i\bar\rho_j|\bar v_i-\bar v_j|^2\dx\ds +\frac{1}{\epsilon}\int_0^t\int_{\T^3}\sum_{i=1}^n\sum_{j=1}^n\bar\theta b_{ij}\rho_i(v_i-\bar v_i)\cdot\bar\rho_j(\bar v_i-\bar v_j)\dx\ds \\
    & -\frac{1}{\epsilon}\int_0^t\int_{\T^3}\sum_{i=1}^n\sum_{j=1}^n\bar\theta b_{ij}\rho_i\rho_j(v_i-v_j)\cdot(v_i-\bar v_i)\dx\ds \\
    & +\frac{1}{\epsilon}\int_0^t\int_{\T^3}\sum_{i=1}^n\sum_{j=1}^n\bar\theta b_{ij}\rho_i\rho_j(v_i-v_j)\cdot(v_i-\bar v_i)\dx\ds =: f_1+\cdots+f_5, 
    \end{split} \] where the last two terms are added and subtracted.

    \noindent
    Now, write the third term as
\[ \begin{split}
    f_3 & = -\frac{1}{\epsilon}\int_0^t\int_{\T^3}\sum_{i=1}^n\sum_{j=1}^n\bar\theta b_{ij}\rho_i(v_i-\bar v_i)\cdot(\rho_j-\bar\rho_j)(\bar v_i-\bar v_j)\dx\ds \\
    & \phantom{xx} +\frac{1}{\epsilon}\int_0^t\int_{\T^3}\sum_{i=1}^n\sum_{j=1}^n\bar\theta b_{ij}\rho_i(v_i-\bar v_i)\cdot\rho_j(\bar v_i-\bar v_j)\dx\ds =: f_{31} + f_{32}
\end{split} \]
\noindent
and combine $f_{32}$ with $f_4$, in order to get \[ \begin{split} f_4 + f_{32} & = -\frac{1}{\epsilon}\int_0^t\int_{\T^3}\sum_{i=1}^n\sum_{j=1}^n\bar\theta b_{ij}\rho_i\rho_j(v_i-\bar v_i)\cdot((v_i-v_j)-(\bar v_i-\bar v_j))\dx\ds \\
& = -\frac{1}{2\epsilon}\int_0^t\int_{\T^3}\sum_{i=1}^n\sum_{j=1}^n\bar\theta b_{ij}\rho_i\rho_j|(v_i-v_j)-(\bar v_i-\bar v_j)|^2\dx\ds. \end{split} \]

Now, we write \[ F_1 = \frac{1}{\epsilon}\int_0^t\int_{\T^3}\sum_{i=1}^n\sum_{j=1}^n\theta b_{ij}\rho_i\rho_jv_i\cdot\bar v_i\dx\ds -\frac{1}{\epsilon}\int_0^t\int_{\T^3}\sum_{i=1}^n\sum_{j=1}^n\theta b_{ij}\rho_i\rho_jv_j\cdot\bar v_i\dx\ds, \]

\[ F_4 = -\frac{1}{\epsilon}\int_0^t\int_{\T^3}\sum_{i=1}^n\sum_{j=1}^n\theta b_{ij}\bar\rho_i\bar\rho_j(\bar v_i-\bar v_j)\cdot\bar v_i\dx\ds, \]

\[ f_1 = -\frac{1}{\epsilon}\int_0^t\int_{\T^3}\sum_{i=1}^n\sum_{j=1}^n \bar \theta b_{ij}\rho_i\rho_j (v_i-v_j)\cdot v_i\dx\ds, \]

\[ f_2 = \frac{1}{\epsilon}\int_0^t\int_{\T^3}\sum_{i=1}^n\sum_{j=1}^n\bar\theta b_{ij}\bar\rho_i\bar\rho_j(\bar v_i-\bar v_j)\cdot\bar v_i\dx\ds, \]

\[ \begin{split}
    f_5 & = \frac{1}{\epsilon}\int_0^t\int_{\T^3}\sum_{i=1}^n\sum_{j=1}^n\bar\theta b_{ij}\rho_i\rho_j v_i\cdot(v_i-\bar v_i)\dx\ds -\frac{1}{\epsilon}\int_0^t\int_{\T^3}\sum_{i=1}^n\sum_{j=1}^n\bar\theta b_{ij}\rho_i\rho_j v_j\cdot(v_i-\bar v_i)\dx\ds,
\end{split} \]

\noindent
so that  
\begin{align*}
    & F_1 + F_4 + f_1 + f_2 + f_5 = \frac{1}{\epsilon}\int_0^t\int_{\T^3}\sum_{i=1}^n\sum_{j=1}^n\theta b_{ij}\rho_i\rho_jv_i\cdot\bar v_i\dx\ds \\
    & -\frac{1}{\epsilon}\int_0^t\int_{\T^3}\sum_{i=1}^n\sum_{j=1}^n\theta b_{ij}\rho_i\rho_jv_j\cdot\bar v_i\dx\ds -\frac{1}{\epsilon}\int_0^t\int_{\T^3}\sum_{i=1}^n\sum_{j=1}^n\theta b_{ij}\bar\rho_i\bar\rho_j(\bar v_i-\bar v_j)\cdot\bar v_i\dx\ds \\
    & -\frac{1}{\epsilon}\int_0^t\int_{\T^3}\sum_{i=1}^n\sum_{j=1}^n \bar \theta b_{ij}\rho_i\rho_j (v_i-v_j)\cdot v_i\dx\ds +\frac{1}{\epsilon}\int_0^t\int_{\T^3}\sum_{i=1}^n\sum_{j=1}^n\bar\theta b_{ij}\bar\rho_i\bar\rho_j(\bar v_i-\bar v_j)\cdot\bar v_i\dx\ds \\
    & +\frac{1}{\epsilon}\int_0^t\int_{\T^3}\sum_{i=1}^n\sum_{j=1}^n\bar\theta b_{ij}\rho_i\rho_j v_i\cdot(v_i-\bar v_i)\dx\ds -\frac{1}{\epsilon}\int_0^t\int_{\T^3}\sum_{i=1}^n\sum_{j=1}^n\bar\theta b_{ij}\rho_i\rho_j v_j\cdot(v_i-\bar v_i)\dx\ds
\end{align*}
and, due to symmetry, 
\[ -\frac{1}{\epsilon}\int_0^t\int_{\T^3}\sum_{i=1}^n\sum_{j=1}^n\theta b_{ij}\rho_i\rho_jv_j\cdot\bar v_i\dx\ds=-\frac{1}{\epsilon}\int_0^t\int_{\T^3}\sum_{i=1}^n\sum_{j=1}^n\theta b_{ij}\rho_i\rho_jv_i\cdot\bar v_j\dx\ds, \] 
which implies that

\begin{align*}
    & F_1 + F_4 + f_1 + f_2 + f_5 = \frac{1}{\epsilon}\int_0^t\int_{\T^3}\sum_{i=1}^n\sum_{j=1}^n\theta b_{ij}\rho_i\rho_jv_i\cdot\bar v_i\dx\ds \\
    & -\frac{1}{\epsilon}\int_0^t\int_{\T^3}\sum_{i=1}^n\sum_{j=1}^n\theta b_{ij}\rho_i\rho_jv_i\cdot\bar v_j\dx\ds -\frac{1}{\epsilon}\int_0^t\int_{\T^3}\sum_{i=1}^n\sum_{j=1}^n\theta b_{ij}\bar\rho_i\bar\rho_j(\bar v_i-\bar v_j)\cdot\bar v_i\dx\ds \\
    & -\frac{1}{\epsilon}\int_0^t\int_{\T^3}\sum_{i=1}^n\sum_{j=1}^n \bar \theta b_{ij}\rho_i\rho_j v_i\cdot v_i\dx\ds +\frac{1}{\epsilon}\int_0^t\int_{\T^3}\sum_{i=1}^n\sum_{j=1}^n \bar \theta b_{ij}\rho_i\rho_j v_i\cdot v_j\dx\ds \\
    & +\frac{1}{\epsilon}\int_0^t\int_{\T^3}\sum_{i=1}^n\sum_{j=1}^n\bar\theta b_{ij}\bar\rho_i\bar\rho_j(\bar v_i-\bar v_j)\cdot\bar v_i\dx\ds +\frac{1}{\epsilon}\int_0^t\int_{\T^3}\sum_{i=1}^n\sum_{j=1}^n\bar\theta b_{ij}\rho_i\rho_j v_i\cdot(v_i-\bar v_i)\dx\ds \\
    & -\frac{1}{\epsilon}\int_0^t\int_{\T^3}\sum_{i=1}^n\sum_{j=1}^n\bar\theta b_{ij}\rho_i\rho_j v_j\cdot(v_i-\bar v_i)\dx\ds
\end{align*}

\noindent
and a rearrangement of the terms gives 
\[ \begin{split}
    F_1 + F_4 + f_1 + f_2 + f_5 & = \frac{1}{\epsilon}\int_0^t\int_{\T^3}\sum_{i=1}^n\sum_{j=1}^n(\theta-\bar\theta) b_{ij}\rho_i\rho_jv_i\cdot\bar v_i\dx\ds \\
    & \phantom{xx} -\frac{1}{\epsilon}\int_0^t\int_{\T^3}\sum_{i=1}^n\sum_{j=1}^n(\theta-\bar\theta) b_{ij}\rho_i\rho_j\bar v_i\cdot \bar v_j\dx\ds \\
    & \phantom{xx} -\frac{1}{\epsilon}\int_0^t\int_{\T^3}\sum_{i=1}^n\sum_{j=1}^n(\theta-\bar\theta) b_{ij}\rho_i\rho_j(v_i-\bar v_i)\cdot\bar v_j\dx\ds \\
    & \phantom{xx} -\frac{1}{\epsilon}\int_0^t\int_{\T^3}\sum_{i=1}^n\sum_{j=1}^n(\theta-\bar\theta) b_{ij}\bar\rho_i\bar\rho_j(\bar v_i-\bar v_j)\cdot\bar v_i\dx\ds \\
    & = \frac{1}{\epsilon}\int_0^t\int_{\T^3}\sum_{i=1}^n\sum_{j=1}^n(\theta-\bar\theta) b_{ij}\rho_i\rho_j(v_i-\bar v_i)\cdot(\bar v_i-\bar v_j)\dx\ds \\
    & \phantom{xx} +\frac{1}{\epsilon}\int_0^t\int_{\T^3}\sum_{i=1}^n\sum_{j=1}^n(\theta-\bar\theta) b_{ij}(\rho_i-\bar\rho_i)\bar\rho_j(\bar v_i-\bar v_j)\cdot\bar v_i\dx\ds \\
    & \phantom{xx} +\frac{1}{\epsilon}\int_0^t\int_{\T^3}\sum_{i=1}^n\sum_{j=1}^n(\theta-\bar\theta) b_{ij}\rho_i(\rho_j-\bar\rho_j)(\bar v_i-\bar v_j)\cdot\bar v_i\dx\ds,
\end{split} \]
so that \[ \begin{split}
    F & = -\frac{1}{2\epsilon}\int_0^t\int_{\T^3}\sum_{i=1}^n\sum_{j=1}^n\bar\theta b_{ij}\rho_i\rho_j|(v_i-v_j)-(\bar v_i-\bar v_j)|^2\dx\ds \\
    & \phantom{xx} -\frac{1}{\epsilon}\int_0^t\int_{\T^3}\sum_{i=1}^n\sum_{j=1}^n\bar\theta b_{ij}\rho_i(v_i-\bar v_i)\cdot(\rho_j-\bar\rho_j)(\bar v_i-\bar v_j)\dx\ds \\
    & \phantom{xx} + \frac{1}{\epsilon}\int_0^t\int_{\T^3}\sum_{i=1}^n\sum_{j=1}^n(\theta-\bar\theta) b_{ij}\rho_i\rho_j(v_i-\bar v_i)\cdot(\bar v_i-\bar v_j)\dx\ds \\
    & \phantom{xx} +\frac{1}{\epsilon}\int_0^t\int_{\T^3}\sum_{i=1}^n\sum_{j=1}^n(\theta-\bar\theta) b_{ij}(\rho_i-\bar\rho_i)\bar\rho_j(\bar v_i-\bar v_j)\cdot\bar v_i\dx\ds \\
    & \phantom{xx} +\frac{1}{\epsilon}\int_0^t\int_{\T^3}\sum_{i=1}^n\sum_{j=1}^n(\theta-\bar\theta) b_{ij}\rho_i(\rho_j-\bar\rho_j)(\bar v_i-\bar v_j)\cdot\bar v_i\dx\ds.
\end{split} \]

\textit{Step 3: Combine $I_4$ with the sixth and eighth term on the right--hand--side of \eqref{1+2+6}:}

\[ \begin{split}
    & I_4+\int_0^t\int_{\T^3}\frac{1}{\bar\theta}\bar\kappa\nabla\bar\theta\cdot\nabla(\theta-\bar\theta)\dx\ds-\int_0^t\int_{\T^3}\frac{1}{\bar\theta^2}\bar\kappa|\nabla\bar\theta|^2(\theta-\bar\theta)\dx\ds \\
    & = \int_0^t\int_{\T^3}\left(\kappa\nabla\log\theta-\bar\kappa\nabla\log\bar\theta\right)\cdot\nabla\bar\theta\dx\ds-\int_0^t\int_{\T^3}(\kappa|\nabla\log\theta|^2-\bar\kappa|\nabla\log\bar\theta|^2)\bar\theta\dx\ds \\
    & \phantom{xx} +\int_0^t\int_{\T^3}\bar\kappa\nabla\log\bar\theta\cdot(\nabla\theta-\nabla\bar\theta)\dx\ds-\int_0^t\int_{\T^3}\bar\kappa|\nabla\log\bar\theta|^2(\theta-\bar\theta)\dx\ds \\
    & = -\int_0^t\int_{\T^3}\bar\theta\kappa|\nabla\log\theta-\nabla\log\bar\theta|^2\dx\ds -\int_0^t\int_{\T^3}(\nabla\log\theta-\nabla\log\bar\theta)\cdot\nabla\bar\theta\kappa\dx\ds \\
    & \phantom{xx} +\int_0^t\int_{\T^3}(\nabla\log\theta-\nabla\log\bar\theta)\cdot\nabla\log\bar\theta\theta\bar\kappa\dx\ds \\
    & = -\int_0^t\int_{\T^3}\bar\theta\kappa|\nabla\log\theta-\nabla\log\bar\theta|^2\dx\ds - \int_0^t\int_{\T^3}(\nabla\log\theta-\nabla\log\bar\theta)\cdot\nabla\log\bar\theta(\theta\bar\kappa-\bar\theta\kappa)\dx\ds.
\end{split} \] 

\textit{Step 4: Combine $I_5$ with the fourth and tenth terms on the right--hand--side of \eqref{1+2+6}:}

\[ \begin{split}
    & I_5-\int_0^t\int_{\T^3}\sum_{i=1}^n\rho_i(v_i-\bar{v}_i)\cdot\bar b_i\dx\ds-\int_0^t\int_{\T^3}\frac{\bar\rho\bar r}{\bar\theta}(\theta-\bar\theta)\dx\ds \\
    & = \int_0^t\int_{\T^3}\sum_{i=1}^n\rho_i(b_i-\bar b_i)\cdot(v_i-\bar v_i)\dx\ds+\int_0^t\int_{\T^3}\left(\frac{\rho r}{\theta}-\frac{\bar\rho\bar r}{\bar\theta}\right)(\theta-\bar\theta)\dx\ds
\end{split} \] and finally,

\textit{Step 5: Combine $I_7$ with the second term on the right--hand--side of \eqref{1+2+6}:}
\[ \begin{split}
    & I_7-\int_0^t\int_{\T^3}\sum_{i=1}^n\frac{\rho_i}{\bar\rho_i}(v_i-\bar v_i)\cdot\bar R_i\dx\ds \\
    & = \int_0^t\int_{\T^3}\sum_{i=1}^n\bar R_i\cdot\bar v_i\dx\ds - \int_0^t\int_{\T^3}\bar Q\dx\ds -\int_0^t\int_{\T^3}\sum_{i=1}^n\frac{\rho_i}{\bar\rho_i}(v_i-\bar v_i)\cdot\bar R_i\dx\ds.
\end{split} \]
Putting everything together, we arrive at \eqref{relenin}.

\end{proof}


\subsection{Validation of the high--friction limit}\label{(3.4)}
A careful estimation of the terms on the right-hand side of \eqref{relenin} implies the following theorem:

\begin{theorem}\label{thm}
     Let $\omega$ be an entropy weak solution of the Class--II model \eqref{intro-mass2}--\eqref{intro-energy2} and $\bar{\omega}$ a strong solution of the Class--I model \eqref{intro-mass1}--\eqref{intro-constr}. We assume that the weak solution satisfies 
     \[ 0\leq\rho_1,\dots,\rho_n\leq M, \quad 0<\gamma\leq\rho\leq M, \quad 0<\gamma\leq\theta\leq M \] and the strong solution satisfies \[ 0<\gamma\leq\bar\rho_1,\dots,\bar\rho_n\leq M, \quad |\bar v_1|,\dots,|\bar v_n|\leq M, \quad 0<\gamma\leq\bar\theta\leq M \] 
     \[ |\nabla\bar v_1|,\dots,|\nabla\bar v_n|\leq M, \quad |\del_t\bar{\theta}|\leq M, \quad |\nabla\bar{\theta}|\leq M \] for some $\gamma,M>0$. Moreover, assume that $\kappa$ and $\frac{\rho r}{\theta}$ are Lipschitz functions of $(\rho_1,\dots,\rho_n,\theta)$, with $\kappa$ bounded away from zero, $b_i$ are Lipschitz functions of $(\rho_1,\dots,\rho_n,v_1,\dots,v_n,\theta)$, for all $i\in\{1,\dots,n\}$ and the free energy functions $\rho_i\psi_i\in C^3(U)$ satisfy \eqref{convexity}, for all $i\in\{1,\dots,n\}$. Then, if the initial data are such that $\mathcal{H}(\omega|\bar\omega)(0)\to0$, as $\epsilon\to0$, we have that $\mathcal{H}(\omega|\bar\omega)(t)\to0$, for all $t>0$, as $\epsilon\to0$.
\end{theorem}

\begin{remark} \label{rem6}
    In the case of the ideal gas, Theorem \ref{thm} is still valid under the additional assumption $0<\gamma\leq\rho_1,\dots,\rho_n\leq M$ (see Remark \ref{rem} or \cite[Section 5]{GeTz23} for more details). 
\end{remark}

\begin{proof}

Having obtained the relative entropy inequality \eqref{relenin}, Theorem \ref{thm} is a direct application of Young's inequality and Gr\"onwall's Lemma. In particular, we estimate each term on the right-hand side of \eqref{relenin}, as follows:

We start by noticing that, according to \cite[Lemma 4.1]{GeTz23}, due to the smoothness of the free energy and the bounds on the strong solution, we have the following bounds: \[ |p_i(\omega|\bar\omega)|\leq C\left(|\rho_i-\bar\rho_i|^2+|\theta-\bar\theta|^2\right) \] and \[ |(\rho_i\eta_i)(\omega|\bar\omega)|\leq C\left(|\rho_i-\bar\rho_i|^2+|\theta-\bar\theta|^2\right), \] which imply that \[ 
    \left|-\int_0^t\int_{\T^3}\sum_{i=1^n}(\rho_i\eta_i)(\omega|\bar\omega)(\del_s\bar\theta+\bar v_i\cdot\nabla\bar\theta)\dx\ds\right|
    \leq C\int_0^t\int_{\T^3}\left(\sum_{i=1}^n|\rho_i-\bar\rho_i|^2+|\theta-\bar\theta|^2\right)\dx\ds \] and \[ 
    \left|-\int_0^t\int_{\T^3}\sum_{i=1^n}p_i(\omega|\bar\omega)\diver\bar v_i\dx\ds\right| \leq C\int_0^t\int_{\T^3}\left(\sum_{i=1}^n|\rho_i-\bar\rho_i|^2+|\theta-\bar\theta|^2\right)\dx\ds. \]

Again by the smoothness of the free energy, and thus the entropy, we obtain \[ |\eta_i-\bar\eta_i|\leq C(|\rho_i-\bar\rho_i|+|\theta-\bar\theta|) \] and thus by Young's inequality, \[ \begin{split}
    & \left|-\int_0^t\int_{\T^3}\sum_{i=1}^n\rho_i(v_i-\bar v_i)\cdot\nabla\bar\theta(\eta_i-\bar\eta_i)\dx\ds\right| \\
    & \phantom{xx} \leq C\int_0^t\int_{\T^3}\left(\sum_{i=1}^n|\rho_i-\bar\rho_i|^2+\sum_{i=1}^n\rho_i|v_i-\bar v_i|^2+|\theta-\bar\theta|^2\right)\dx\ds.
\end{split} \]

Moreover, by Young's inequality and the Lipschitz continuity of $b_i$, \[ \begin{split}
    & \left|\int_0^t\int_{\T^3}\sum_{i=1}^n\rho_i(b_i-\bar b_i)\cdot(v_i-\bar v_i)\dx\ds\right| \leq C\int_0^t\int_{\T^3}\sum_{i=1}^n\left(\rho_i|v_i-\bar v_i|^2+\rho_i|b_i-\bar b_i|^2\right)\dx\ds \\
    & \phantom{xxxxxxxxxxx} \leq C\int_0^t\int_{\T^3}\left(\sum_{i=1}^n|\rho_i-\bar\rho_i|^2+\sum_{i=1}^n\rho_i|v_i-\bar v_i|^2+|\theta-\bar\theta|^2\right)\dx\ds.
\end{split} \]

Furthermore, \[ \begin{split}
    & \left|-\int_0^t\int_{\T^3}(\nabla\log\theta-\nabla\log\bar\theta)\cdot\nabla\log\bar\theta(\theta\bar\kappa-\bar\theta\kappa)\dx\ds\right| \\
    & \phantom{xxxx}\leq \int_0^t\int_{\T^3}|\sqrt{\bar\theta}\sqrt{\kappa}(\nabla\log\theta-\nabla\log\bar\theta)\cdot\nabla\log\bar\theta(\kappa-\bar\kappa)\frac{\theta}{\sqrt{\bar\theta\kappa}}|\dx\ds \\
    & \phantom{xxxxxx} + \int_0^t\int_{\T^3}|\sqrt{\bar\theta}\sqrt{\kappa}(\nabla\log\theta-\nabla\log\bar\theta)\cdot\nabla\log\bar\theta(\theta-\bar\theta)\frac{\sqrt{\kappa}}{\sqrt{\bar\theta}}|\dx\ds \\
    & \phantom{xxxx}\leq \frac{1}{2}\int_0^t\int_{\T^3}\bar\theta\kappa|\nabla\log\theta-\nabla\log\bar\theta|^2\dx\ds + C\int_0^t\int_{\T^3}\left(\sum_{i=1}^n|\rho_i-\bar\rho_i|^2+|\theta-\bar\theta|^2\right)\dx\ds
\end{split} \] by Young's inequality, the lower bounds of $\bar\theta$ and $\kappa$ and the Lipschitz continuity of $\kappa$.

Also, by \eqref{R_i} $\bar R_i = \mathcal{O}(\epsilon)$ and $\sum_i\bar R_i = \mathcal{O}(\epsilon^2)$ and by \eqref{Q} $\bar{Q} = \mathcal{O}(\epsilon^2)$. Thus, \[ \sum_{i=1}^n\bar R_i\cdot \bar v_i-\bar Q = \sum_{i=1}^n\bar R_i\cdot\bar u_i+\bar v\sum_{i=1}^n \bar R_i-\bar Q = \mathcal{O}(\epsilon^2) \] i.e.
\[ \left|\int_0^t\int_{\T^3}\left(\sum_{i=1}^n\bar R_i\cdot \bar v_i-\bar Q\right)\dx\ds\right|\leq\mathcal{O}(\epsilon^2) \] and \[ \begin{split}
    & \left|-\int_0^t\int_{\T^3}\sum_{i=1}^n\frac{\rho_i}{\bar\rho_i}(v_i-\bar v_i)\cdot\bar R_i\dx\ds\right| \\
    & \phantom{xxxxxxx} \leq C \left(\int_0^t\int_{\T^3}\sum_{i=1}^n\rho_i|v_i-\bar v_i|^2\dx\ds+\int_0^t\int_{\T^3}\sum_{i=1}^n\rho_i\frac{\bar R_i^2}{\bar\rho_i^2}\dx\ds\right) \\
    & \phantom{xxxxxxx} \leq C \int_0^t\int_{\T^3}\sum_{i=1}^n\rho_i|v_i-\bar v_i|^2\dx\ds+\mathcal{O}(\epsilon^2).
\end{split} \]

Finally, \[ \begin{split}
    & \left|-\frac{1}{\epsilon}\int_0^t\int_{\T^3}\sum_{i=1}^n\sum_{j=1}^n\bar\theta b_{ij}\rho_i(v_i-\bar v_i)\cdot(\rho_j-\bar\rho_j)(\bar v_i-\bar v_j)\dx\ds\right| \\
    & \phantom{xxxx} \leq C\int_0^t\int_{\T^3}\left(\sum_{i=1}^n|\rho_i-\bar\rho_i|^2+\sum_{i=1}^n\rho_i|v_i-\bar v_i|^2\right)\dx\ds
\end{split} \] and $C$ does not depend on $\epsilon$, because $\frac{1}{\epsilon}(\bar v_i-\bar v_j)= \frac{1}{\epsilon}(\bar u_i-\bar u_j)=\mathcal{O}(1)$ and the remaining terms are treated in a similar fashion.

Putting everything together, we obtain 
\begin{equation}\label{final}
    \begin{split}
    & \mathcal{H}(\omega|\bar\omega)(t) + \frac{1}{2\epsilon}\sum_{i=1}^n\sum_{j=1}^n\int_0^t\int_{\T^3}\bar{\theta}b_{ij}\rho_i\rho_j|(v_i-v_j)-(\bar{v}_i-\bar{v}_j)|^2\dx\ds \\
    & + \frac{1}{2}\int_0^t\int_{\T^3}\bar{\theta}\kappa|\nabla\log\theta-\nabla\log\bar{\theta}|^2\dx\ds \leq \mathcal{H}(\omega|\bar\omega)(0) \\
    & +C \int_0^t\int_{\T^3}\left(\sum_{i=1}^n|\rho_i-\bar\rho_i|^2+\sum_{i=1}^n\rho_i|v_i-\bar v_i|^2+|\theta-\bar\theta|^2\right)\dx\ds+\mathcal{O}(\epsilon^2),
\end{split}
\end{equation} 
where by virtue of \eqref{relen} and \eqref{bound1}, 
\[ \int_0^t\int_{\T^3}\left(\sum_{i=1}^n|\rho_i-\bar\rho_i|^2+\sum_{i=1}^n\rho_i|v_i-\bar v_i|^2+|\theta-\bar\theta|^2\right)\dx\ds \leq C \int_0^t \mathcal{H}(\omega|\bar\omega)(s)\ds. \] 
The dissipation terms on the left-hand side of \eqref{final} are non-negative and thus can be neglected, yielding \[ \mathcal{H}(\omega|\bar\omega)(t) \leq [\mathcal{H}(\omega|\bar\omega)(0)+\mathcal{O}(\epsilon^2)]+C\int_0^t \mathcal{H}(\omega|\bar\omega)(s)ds, \] where $C>0$ is independent of $\epsilon$.

By Gr\"onwall's Lemma \[ \mathcal{H}(\omega|\bar\omega)(t) \leq [\mathcal{H}(\omega|\bar\omega)(0)+\mathcal{O}(\epsilon^2)]e^{Ct}, \] where $C>0$ does not depend on $\epsilon$. Letting $\epsilon\to0$, $\mathcal{H}(\omega|\bar\omega)(0)\to0$ and thus $\mathcal{H}(\omega|\bar\omega)(t)\to0$, for all $t>0$ and the proof is completed.

\end{proof}


\smallskip
{\bf Acknowledgments.} We would like to thank the anonymous referee for his/her many valuable comments. 

{\bf Funding.} Research partially supported by King Abdullah University of Science and Technology (KAUST) baseline funds. The first author acknowledges partial support from the Austrian Science Fund (FWF), grants P33010 and F65. This work has received funding from the European Research Council (ERC) under the European Union's Horizon 2020 research and innovation programme, ERC Advanced Grant no. 101018153.

\smallskip
{\bf Conflict of interest.} The authors have no conflict of interest to report.

	

\begin{thebibliography}{11}

\bibitem{BoDr15} 
\textnormal{\textsc{D. Bothe, W. Dreyer}, Continuum Thermodynamics of chemically reacting fluid mixtures, {\it Acta Mech.} {\bf 226} (2015), 1757--1805.}

\bibitem{BoGrPa19} 
\textnormal{\textsc{L. Boudin, B. Grec, V. Pavan}, Diffusion models for mixtures using a stiff dissipative hyperbolic formalism, {\it J. of Hyp. Dif. Equat.} {\bf 16} (2019), 293--312.}

\bibitem{Callen85} 
\textnormal{\textsc{H. Callen}, {\it Thermodynamics and an introduction to thermostatistics}, New York: John Wiley \& Sons (1985).}

\bibitem{CLL94}
\textnormal{\textsc{ G.-Q. Chen, C.D. Levermore, and T.-P. Liu}, Hyperbolic conservation laws
   with stiff relaxation terms and entropy, {\it Comm. {P}ure and {A}ppl. {M}ath.}
   \textbf{47} (1994), 787--830.}

\bibitem{ChTz18} 
\textnormal{\textsc{C. Christoforou, A.E. Tzavaras}, Relative entropy for hyperbolic–parabolic systems and application to the constitutive
theory of thermoviscoelasticity, {\it Arch. Rational Mech. Anal.} {\bf 229} (2018) 1--52.}

\bibitem{Da79} 
\textnormal{\textsc{C. M. Dafermos}, Stability of motions of thermoelastic fluids, {\it J. Thermal Stresses} {\bf 2} (1979), 127--134.}

\bibitem{GeTz23} 
\textnormal{\textsc{S. Georgiadis, A.E. Tzavaras},
Asymptotic derivation of multicomponent
compressible flows with heat conduction and mass diffusion. 
{\em ESAIM: Math. Model. Numer. Anal.} 57 (2023), 69---106.}

\bibitem{Gi91}
\textnormal{\textsc{V. Giovangigli}, Convergent iterative methods for multicomponent diffusion,
{\it Impact of Computing in Science and Engineering} {\bf 3} (1991), 244--276.}

\bibitem{Gi99} 
\textnormal{\textsc{V. Giovangigli}, {\it Multicomponent Flow Modeling}, Birkh\"auser Boston 1999.}

\bibitem{HL92} 
\textnormal{\textsc{ L. Hsiao, T.-P. Liu}, 
Convergence to nonlinear diffusion waves for solutions of a system of hyperbolic conservation laws with damping.
{\it Comm. Math. Phys.} {\bf 143} (1992), 599--605.}

\bibitem{HuJuTz19}
\textnormal{\textsc{X. Huo, A. J\"ungel and  A.E. Tzavaras}, High-friction limits of Euler flows for multicomponent systems, {\it Nonlinearity} {\bf 32} (2019), 2875--2913.}

\bibitem{JX95} 
\textnormal{\textsc{ S. Jin, Z. Xin},  The relaxation schemes for 
systems of conservation laws in arbitrary space dimensions,
{\it Comm.\ Pure Appl.\ Math.} {\bf 48} (1995), 235--276.}


\bibitem{JPT98} 
\textnormal{\textsc{ S. Jin, L. Pareschi and G. Toscani}, 
Diffusive relaxation schemes for multiscale discrete-velocity kinetic equations.
{\it SIAM J. Numer. Anal.} {\bf 35} (1998), 2405--2439. }

\bibitem{YYZ15} 
\textnormal{\textsc{ Z.~Yang, W.-A.~Yong and Y.~Zhou}. A rigorous derivation of
multicomponent diffusion laws. Preprint, 2015. {\tt arXiv:1502.03516}. }


\bibitem{WeKr00}
\textnormal{\textsc{J.A. Wesselingh, R. Krishna},
Mass Transfer in Multicomponent Mixtures, (Delft: Delft
University Press) 2000.}

\end{thebibliography}
\end{document}